\documentclass[onefignum,onetabnum]{siamart190516}

\usepackage[english]{babel}

\usepackage{lipsum}
\usepackage{amsfonts}
\usepackage{graphicx}
\usepackage{epstopdf}
\usepackage{algorithmic}
\ifpdf
  \DeclareGraphicsExtensions{.eps,.pdf,.png,.jpg}
\else
  \DeclareGraphicsExtensions{.eps}
\fi
\usepackage{algorithmic} 
\Crefname{ALC@unique}{Line}{Lines} 

\newsiamremark{remark}{Remark}
\newsiamthm{problem}{Problem}
\newsiamthm{example}{Example}
\newsiamthm{claim}{Claim}

\title{Optimizing Static Linear Feedback: Gradient Method\thanks{Submitted to the editors on April 2020. \funding{The revised version of this work was funded by Russian Science Foundation under Grant 21-71-30005.}}}

\author{Ilyas Fatkhullin${}^\ddagger$\thanks{Moscow Institute of Physics and Technology, Institutskiy per. 9, 141700, Moscow Region, Dolgoprudny, Russia (\email{ilyas.fn979@gmail.com}).}
\and Boris Polyak\thanks{Institute for Control Sciences, Profsoyuznaya 65, 117806, Moscow, Russia
  (\email{boris@ipu.ru}).}
  }

\usepackage{amsopn}

\usepackage{fullpage}

\usepackage{url}

\makeindex

\begin{document}

\maketitle
\begin{abstract}
The linear quadratic regulator is the fundamental problem of optimal control. Its state feedback version was set and solved in the early 1960s. However the static output feedback problem has no explicit-form solution. It is suggested to look at both of them from another point of view as matrix optimization problems, where the variable is a feedback matrix gain. The properties of such a function are investigated, it turns out to be smooth, but not convex, with possible non-connected domain.  Nevertheless, the gradient method for it with the special step-size choice converges to the optimal solution in the state feedback case and to a stationary point in the output feedback case. The results can be extended for the general framework of unconstrained optimization and for reduced gradient method for minimization with equality-type constraints.    
\end{abstract}



\section{Introduction}\label{sec:intro}
The linear quadratic regulator (LQR) problem is formulated as an optimization problem of minimizing a quadratic integral cost with respect to control function. It has been extensively analyzed in the last century since the seminal works of Kalman in 1960 \cite{Kalman,Kalman2}. The main result claims that for an infinite-horizon LTI system the optimal control can be expressed as linear static state feedback. The optimal gain can be found by solving the algebraic matrix Riccati equation (ARE). The results became classical and were immediately included in textbooks on control \cite{Anderson,Athans_Falb,Kwakernaak}. New approaches to the problem were based on the techniques of semidefinite programming --- reduction to convex optimization with Linear Matrix Inequalities (LMIs) as constraints \cite{Boyd_LMI,Iwasaki,balandin_synthesis_2007,Khlebnikov_new_solution}. Linear static feedback is a very natural and simple form of control for engineers, thus there were many attempts to extend the technique for other control problems. 

The nearest relative of LQR is output feedback --- the same LTI system with quadratic performance in the case when full state is not measured but some output (a linear function of the state) is available. The attempts to apply static output feedback (SOF) met numerous difficulties. The problem was first addressed by Levine and Athans \cite{Levine_Athans}, but it was discovered that such stabilizing control may be lacking and there are no simple optimality certificates if it does exist. Serious theoretical efforts were directed on the formulation of existence conditions, see  \cite{Syrmos_survey,np_hardness}, but the problem remains open. If a system is stabilizable via a static output controller, there are just necessary conditions for optimality; moreover, these conditions are formulated as a system of nonlinear matrix equations \cite{Levine_Athans}. Thus the design of optimal SOF implies application of numerical methods. The first one was proposed in \cite{Levine_Athans}, but it requires to solve nonlinear matrix equations on each iteration. The method suggested by Anderson and Moore \cite{Anderson} is based on the solution of linear matrix equations only, but its properties were not obvious. Some results on the convergence of both methods can be found in \cite{Makila_Toivonen_Survey}. Since then, numerous  iterative schemes have been proposed, see \cite{Toivonen,Makila_Toivonen_Survey,Mostafa_CG,Moerder,eilbrecht_distributed_IFAC_2017,rautert_computational_1997,Iwasaki,Martensson_2009} and references therein. However rigorous validation is lacking for many of them, while some others include hard nonlinear problems to be solved at each iteration. To sum up, optimization of SOF remains a challenging problem.

A promising tool for solving both state and output feedback control is the direct gradient method. Matrix  gain \(K\)  for state \(u(t)=Kx(t)\) or output \(u(t)=Ky(t)\) control is considered as variable for optimization of the objective function which is expressed as \(f(K)\). This function is well-defined for the set of stabilizing controllers \(\mathcal{S}\) (otherwise the quadratic integral performance index is not defined). The set \(\mathcal{S}\) is open and the minimum of \(f(K)\) is achieved at the interior point. Thus a simple gradient method for unconstrained minimization of \(f(K)\) can be applied
\[K_{j+1}=K_j-\gamma_j\nabla f(K_j)\]
provided that the initial stabilizing controller $K_0$ is known. Gradient \(\nabla f(K)\) for state feedback case has been found in the pioneering paper of Kalman \cite{Kalman}, for output feedback it was obtained by Levine and Athans \cite{Levine_Athans}. Its calculation is computationally inexpensive --- it requires the solution of two Lyapunov equations. Such approach looks very attractive, but there are some obstacles. For state feedback the set \(\mathcal{S}\) is connected but (in general) nonconvex \cite{Ackermann}, thus \(f(K)\) can be nonconvex as well. A more sophisticated situation is met for output control. The set \(\mathcal{S}\) can be disconnected \cite{Gryazina_2008,Lavaei_exponential_2019} while saddle points or local minima can exist in a connected component. These difficulties explain why in many papers gradient method was applied without rigorous validation, as a purely heuristic algorithm. Luckily it worked successively in many applications. 

Recently there was a breakthrough in this field. First there  appeared papers devoted to discrete-time version of state-feedback LQR \cite{LQR_discrete,LQR_discrete_Fazel}. \(f(K)\), despite being non-convex is shown to satisfy the so-called Lezanski-Polyak-Lojasiewicz (LPL) condition. This condition was proposed in the works \cite{Lezanski, LPL_Polyak, Lojasiewicz} back in the 1960s and still remains a powerful tool in non-convex optimization \cite{PL_Karimi}. Based on the LPL condition it was possible to prove global convergence of the gradient method to optimal controller. Important works \cite{LQR_continuous,Mohammadi} overcome the nonconvexity obstacle for classical continuous-time LQR. It was proved that the LPL condition holds for this case and the gradient method converges. This line of research is continued in \cite{Furieri_2019,Jansch-Porto_2020,Hu_2019,Bu2020GlobalCO}.

The situation is more complicated for output control. As we mentioned above, the domain \(\mathcal{S}\) can be nonconnected, and values of local minima at different connected components are different. Moreover, several local minima points can exist in a single component. Thus it is hard to expect something better than convergence to a stationary point.

\textbf{Contributions of the paper} As we have mentioned, direct optimization methods for feedback control is a highly intensive direction of recent research. If compared with known results the main contributions of the presented paper can be formulated as follows.

a. Most of the results on the convergence of the gradient method for \textit{state  feedback} were known for \textit{discrete-time} case  \cite{LQR_discrete,LQR_discrete_Fazel,Hu_2019,Jansch-Porto_2020}. We focus on the \textit{continuous-time} case and prove convergence of the method to the single minimizer with a linear rate. Similar results have been obtained in
\cite{LQR_continuous,Mohammadi} but the technique of the proof there is completely different. In \cite{LQR_continuous} the problem  was converted into convex optimization by the change of variables. However such transformation is possible for state control only, while we use the technique which fits for both state and output cases.

b. Novel results on convergence (and rate of convergence) to stationary points are obtained for \textit{output feedback}. They are based on the proved $L$-\textit{smoothness} of the objective function. It is worth mentioning that a similar analysis can be applied to the  wider class of problems which is called in \cite{Makila_Toivonen_Survey} \textit{parametric LQR}. It includes such important problems as \textit{low-order control, PID control, decentralized control}.

c. The particular properties of the feedback optimization allow to design new versions of minimization methods --- such as novel step-size rule for gradient and conjugate gradient methods or global convergence of the reduced gradient method. These algorithms can be extended to a \textit{general optimization setup}, see \cref{sec:Future_work}.

\textbf{Organization of the paper}
In \cref{sec:Problem_Statement} we formulate the LQR as a matrix optimization problem with nonlinear equality constraints. Then it is reduced to matrix unconstrained minimization with objective \(f(K)\) and its domain \(\mathcal{S}\). \Cref{sec:Properties} discusses the properties of this function defined on a generally non-convex set. The most important are \(L\)-smoothness property; for state feedback case LPL condition holds.  In \cref{sec:Methods} the gradient flow on this set is showed to be exponentially stable and the discrete gradient method with special step-size rule is introduced. The convergence guarantees are presented. \Cref{sec:Simulation} illustrates the numerical experiments for the proposed method. In  \cref{sec:Future_work} we address the links between the proposed method and general optimization problems such as unconstrained smooth minimization and optimization with equality-type constraints. Finally, in \cref{sec:Conclusion} we discuss directions for future research. The proofs of the results are relegated to Appendix.  

\section{Problem Statement}\label{sec:Problem_Statement}

We use standard notation:
 \(\|\cdot\| \--\) spectral norm of a matrix;
 \(\|\cdot\|_F \--\) its Frobenius norm;
\(\mathbb{S}_{n} \--\) the set of symmetric matrices; $I$ is the identity matrix;
 \(A \succ B\) (\(A\succeq B \)) means that the matrix \(A-B\) is positive (semi-)definite;
the eigenvalues \(\lambda_{i}(A)\) of a matrix \(A \in \mathbb{R}^{n\times n}\) are indexed in an increasing order with respect to their real parts, i.e., $
\Re\left(\lambda_{1}(A)\right) \leq \ldots \leq \Re\left(\lambda_{n}(A)\right).
$

Consider linear time-invariant system
\begin{align*}
&\dot{x}(t) = A x(t)+B u(t)  \stepcounter{equation}\tag{\theequation}\label{LTI}, \\
&y(t) = Cx(t), \\
&\mathbb{E}x(0)x(0)^{\top} = \Sigma,
\end{align*}
with state $x$ and output $y$ and matrices \(A \in \mathbb{R}^{n \times n}\), \(B \in \mathbb{R}^{n \times m}, \) \(C \in \mathbb{R}^{r \times n}\). The infinite-horizon LQR performance criterion is given by
\begin{equation}
\mathbb{E}\int_{0}^{\infty}\left[x(t)^{T} Q x(t)+u(t)^{T} R u(t)\right] d t,
\end{equation}
where the expectation is taken over the distribution of an initial condition \(x(0)\) with zero mean and covariance matrix \(\Sigma\), and the quadratic cost is parameterized by \(0 \prec Q \in \mathbb{S}_{n},\) and \(0 \prec R \in \mathbb{S}_{m} .\) 

The static  feedback control is \(u(t) = - Ky(t)\), where gain \(K\in \mathbb{R}^{m \times r}\) is a constant matrix. Then the closed loop system is given by
\begin{equation}
\dot{x}(t) = A_K x(t), \, A_K=(A - B K C)
\end{equation}
and objective function becomes
\begin{equation}
f (K)=\mathbb{E}\int_{0}^{\infty}\left[x(t)^{\top} (Q + C^{\top} K^{\top} R K C) x(t)\right] d t.
\end{equation}
We use notation $f(K)$ to underline that the performance index depends on gain only; all other ingredients of system description are known. Thus our optimization problem is
\begin{equation}\label{problem}
f(K)\rightarrow \min_{K\in\mathcal{S}},
\end{equation}
here \(\mathcal{S}\) is the set of stabilizing feedback gains,
\[
\mathcal{S}=\left\{K \in \mathbb{R}^{m \times n}: \Re\lambda_{i}(A-B K C)<0, \forall i \in \overline{1,n}\right\}.
\]
Indeed, $f(K)$ is defined for stabilizing controllers $K\in \mathcal{S}$ only.

The problem of existence of stable output feedback is hard, see e.g. \cite{np_hardness,Syrmos_survey}. However we are not interested in this, our main assumption is that a stabilizing controller exists and is available:

\textbf{$K_0\in\mathcal{S}$ is known. } 

For instance if $A$ is Hurwitz then we can take $K_0=0$. This controller will be taken as the initial approximation for iterative methods. Thus our goal is to improve the performance of the known regulator.
Denote
 \(\mathcal{S}_0 \)  the sublevel set 
\[
\mathcal{S}_0 = \left\{K \in \mathcal{S}: f(K) \leq f(K_0)\right\}.
\]

We suppose the following
\textbf{Assumptions} hold:
\begin{itemize}
\item \(K_0 \in \mathcal{S}\) exists;
\item \(Q, R, \Sigma \succ 0\);
\item \(rank(C) = r\).
\end{itemize}
Notice that there are no assumtions on controllability/observability, existence of \(K_0 \in \mathcal{S}\) suffices. Also we assume $B\neq 0$, otherwise the problem is trivial. Condition \(Q\succ 0\) sometimes can be relaxed to \(Q\succeq 0\), but we do not focus on this. 

We distinguish two main versions of the problem:

1. \textbf{SLQR} - \textit{state LQR }- if $C=I$, that is the state $x(t)$ is available as control input. If it is needed to specify the performance index $f(K)$ for this case, we denote it as $f_S(K)$.

2. \textbf{OLQR} - \textit{output LQR} - if $C\neq I$, when output $y(t)$ is the only information available. We use notation $f_O(K)$ to specify this case, while $f(K)$ is used in general situation.

Let us formulate the problem as matrix constrained optimization one. To avoid calculation of integrals Bellman lemma \cite{bellman_notes_1957} is instrumental. 
\begin{lemma}\label{Bellman}
Given \( W \succ 0,\) and a Hurwitz matrix \(A\). Then on the solution of the LTI system
\[
\dot{x}(t)=A x(t), \quad x(0)=x_{0}
\]
it holds that
\[
\int_{0}^{\infty} x^{\top}(t) W x(t) d t = x_{0}^{\top} X x_{0} ,
\]
where \(X\) is the solution of the Lyapunov matrix equation
\[
A^{\top} X+X A=-W.
\]
\end{lemma}
Applying this result we rewrite \ref{problem} in the final form
\begin{problem}
\begin{equation}
f(K):= \operatorname{Tr}\left(X \Sigma\right) \rightarrow \min_{K},
\label{eq:LQR_cost}
\end{equation}
\begin{equation}\label{eq:Lyapunov_eq_X}
(A-B K C)^{\top} X + X (A-B K C)+C^{\top} K^{\top} R K C+Q=0, \, X\succ 0.
\end{equation}
\end{problem}
This is an optimization problem with matrix variables $K, X$ and nonlinear equality-type constraint (\ref{eq:Lyapunov_eq_X}). For $K\in \mathcal{S}$ the solution $X\succ 0$ of this equation exists (Lyapunov theorem), we denote it as $X(K)$. Thus the problem is rewritten in the form (\ref{problem}) with $f(K)= \operatorname{Tr}\left(X(K) \Sigma\right)$.

In the next section we analyse the properties of the function $f(K)$, its domain $\mathcal{S}$ and sublevel set $\mathcal{S}_0$.

\section{Properties of \(f(K)\)}\label{sec:Properties}
\subsection{Examples} We start with few simple examples to exhibit the variety of situations.

\begin{example}\label{example:1d_simple}
	Let us consider 1D example with parameters \(A = 0 \in \mathbb{R}, Q = R = 2B = 1 \in \mathbb{R}, K = k \in \mathbb{R}\).  The function 
	\[f(k) = k + \frac{1}{k}\] can be written explicitly.
\end{example}
Here $\mathcal{S}=\mathbb{R}_+$ is convex and unbounded, $\mathcal{S}_0$ is bounded, $f(K)$ is convex and unbounded on $\mathcal{S}$ (see \cref{fig:one_d_ex}).
	\begin{figure}[!]
	\centering
	\includegraphics[scale=0.35]{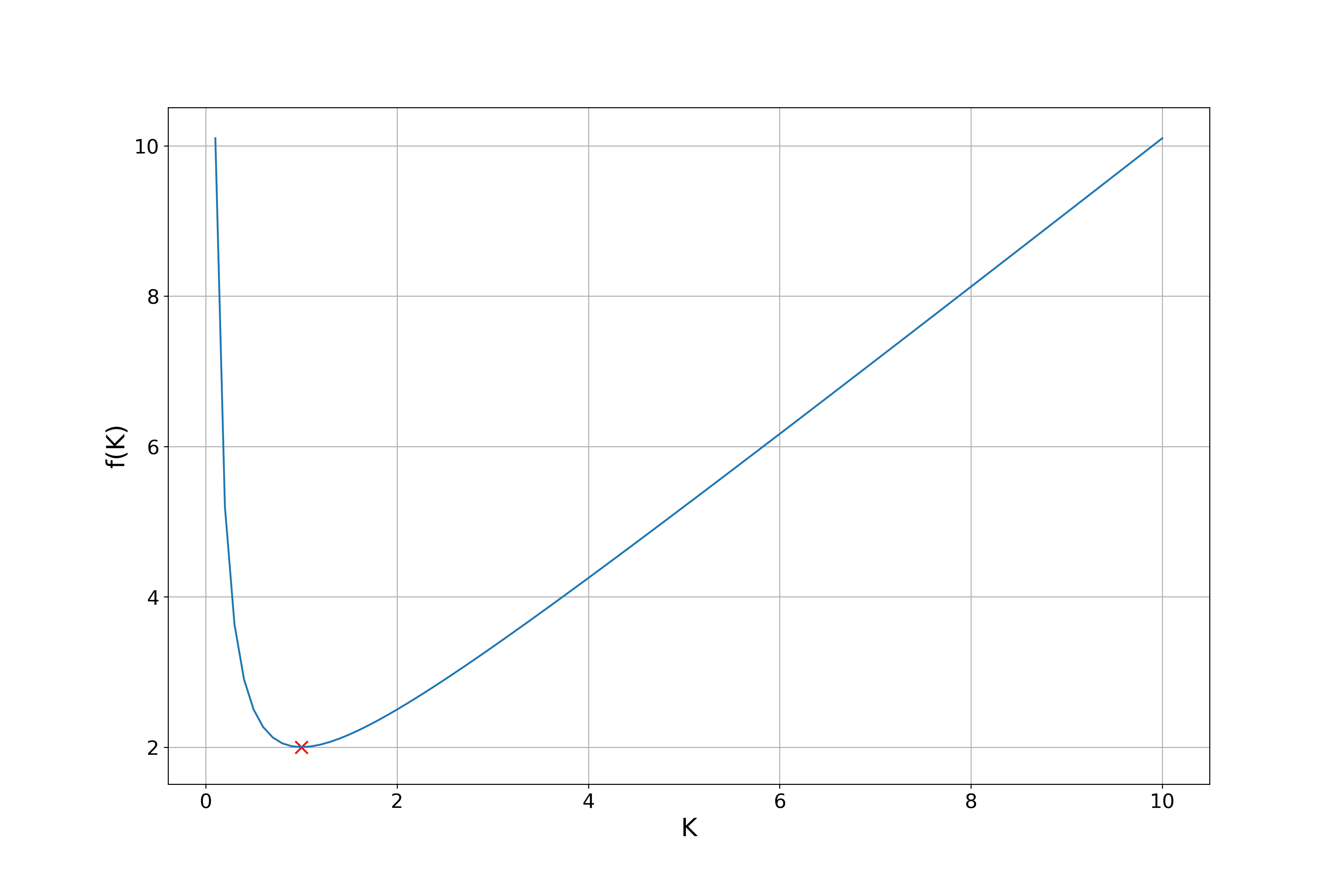}
	\caption{\(f(K)\) for 1D example}
	\label{fig:one_d_ex}
\end{figure}

\begin{example}\label{example:2}
  Let \(n=2, m=2\), set \(A, B\) and \(C\) to be identity matrices. Then 
  \[\mathcal{S} = \left\{K \in \mathbb{R}^{2\times 2}: k_{11}+k_{22} < 1+k_{11}k_{22} + k_{12}k_{21}, k_{11}+k_{22}<2 \right\}.\]
\end{example}
We see that $\mathcal{S}$ is not convex. This can be verified if one takes a cut \(x = k_{11} = k_{12}, y = k_{22} = k_{21} \) (see \cref{fig:nonconvexity_4}). Moreover the boundary of $\mathcal{S}$ is non-smooth.

 \begin{example}\label{example:3}
  For \(n=3, m=1\) consider the matrices
  \(A = \begin{pmatrix} 0 & 1 & 0\\ 0 & 0 & 1\\0 & 0& 0 \end{pmatrix}, B = \begin{pmatrix} 0 \\ 0 \\ 1 \end{pmatrix}\) and \(C = I\). Then 
    \[\mathcal{S} = \left\{K \in \mathbb{R}^{1\times 3}:  k_1 > 0, k_2 k_3 > k_1\right\}.\]
\end{example}    
Again $\mathcal{S}$ is not convex with non-smooth boundary. For instance, the cut \(x = k_1, y = k_2 = k_3\) (see \cref{fig:nonconvexity_3}) is not convex.
 \begin{figure}[!]
	\begin{minipage}{0.48\textwidth}
		\includegraphics[width=1\linewidth]{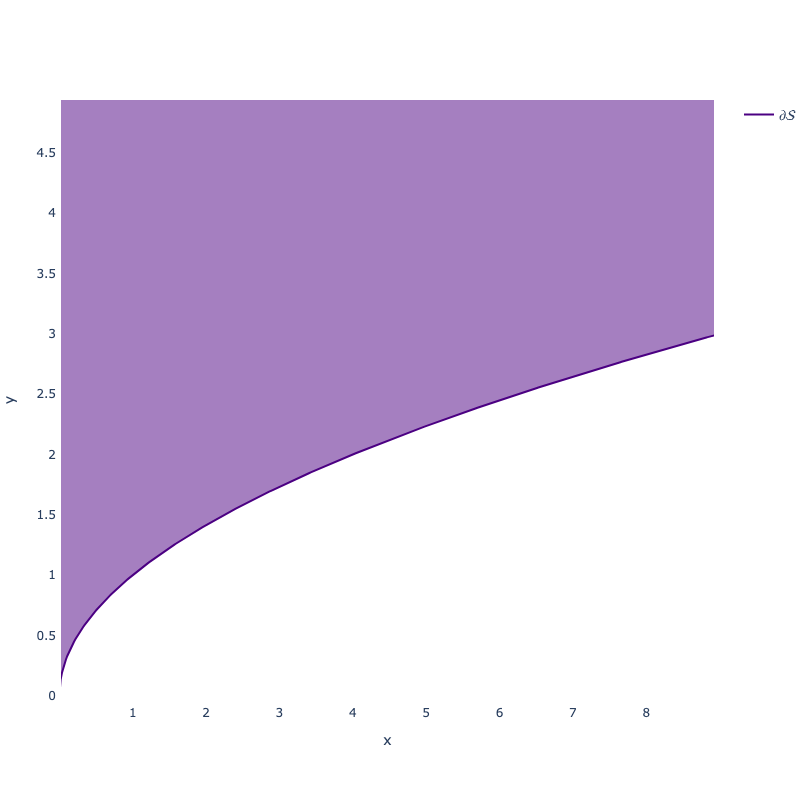}
		\caption{Nonconvex cut of $ \mathcal{S}$ for \(m=n=2\)}
		\label{fig:nonconvexity_4}
	\end{minipage}\hfill
	\begin{minipage}{0.48\textwidth}
		\centering
		\includegraphics[width=1\linewidth]{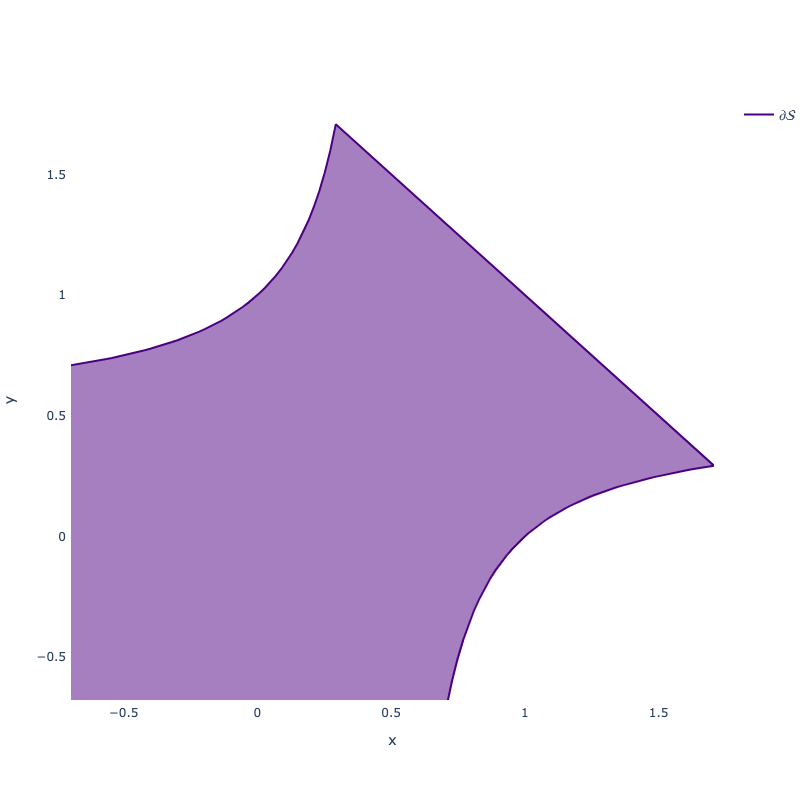}
		\caption{Nonconvex cut of $ \mathcal{S}$ for \(m=1, n=3\)}
		\label{fig:nonconvexity_3}
	\end{minipage}
\end{figure}

Previous examples related to SLQR (state feedback). Now we proceed to OLQR (output feedback).

\begin{example}\label{example:1d_two_local}
 Consider an example with a scalar control and \(Q = I_3 , R = 1\),
 \(A = \begin{pmatrix} 0 & 1 & 0\\ 0 & 0 & 1\\-1 & -1& -\alpha \end{pmatrix},\) \(B = \begin{pmatrix} 0 \\ 0 \\ 1 \end{pmatrix}\) and \(C = \begin{pmatrix} 5 & 2 & 1 \end{pmatrix}\).
 \end{example}
Then
 \[\mathcal{S} = \left\{k \in \mathbb{R}: k+\alpha>0, (k+\alpha)(2k+1)>5k+1>0\right\}.\]
If \(\alpha = -1\) this set is non-connected, it has two connectivity components. The function is illustrated on \cref{fig:non-connected}. It has a single minima at each of the components.  If \(\alpha = -1.4\) this set is connected, it is a ray \(k>-0.2\). The function is illustrated on \cref{fig:1d_two_local}. It has two local minima located in the same connected component. 

 \begin{figure}[!]
	\begin{minipage}{0.48\textwidth}
		\includegraphics[width=1\linewidth]{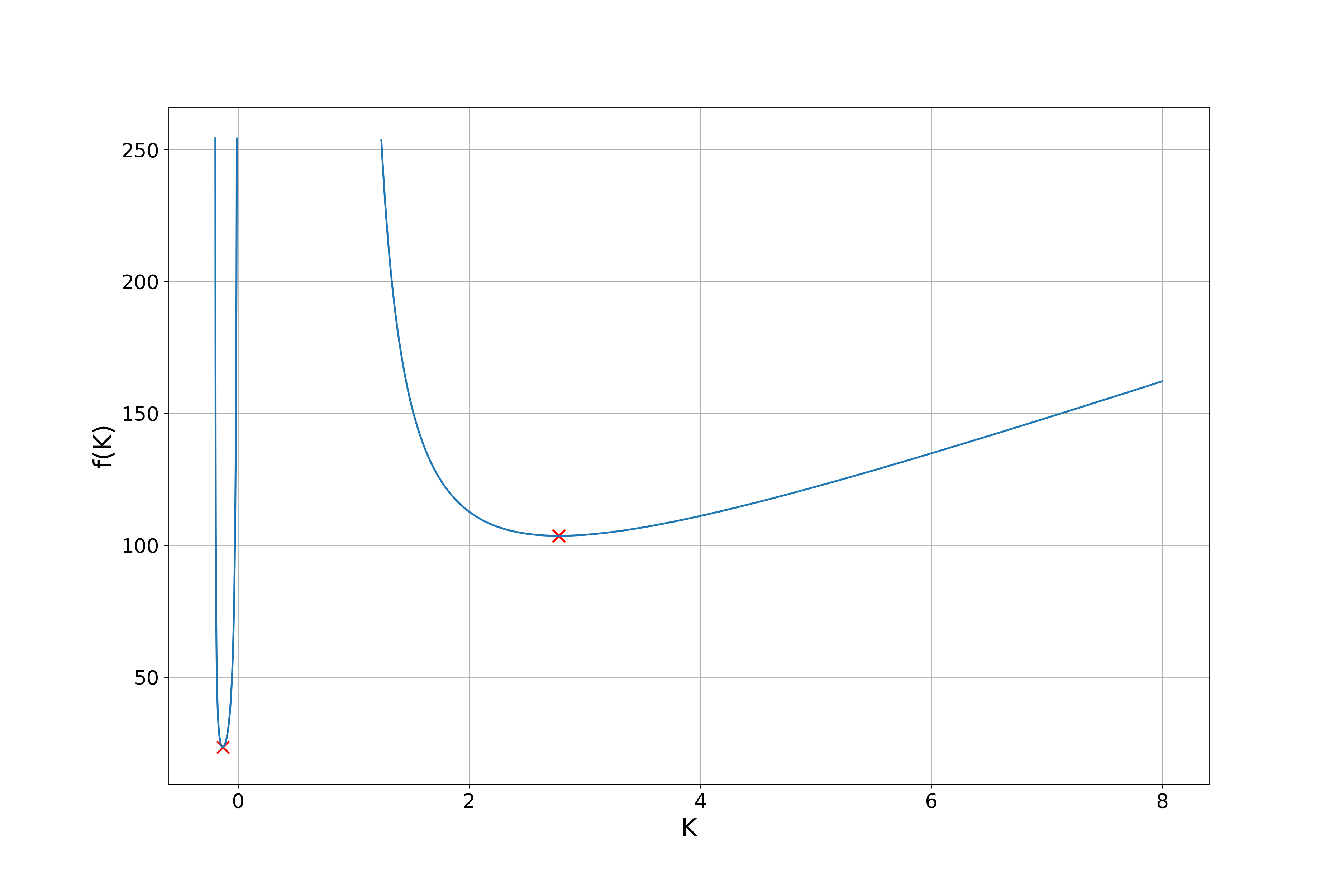}
		\caption{$f_{O}(K)$ for scalar output control with local minima in two disconnected set.}
	\label{fig:non-connected}
	\end{minipage}\hfill
	\begin{minipage}{0.48\textwidth}
		\centering
		\includegraphics[width=1\linewidth]{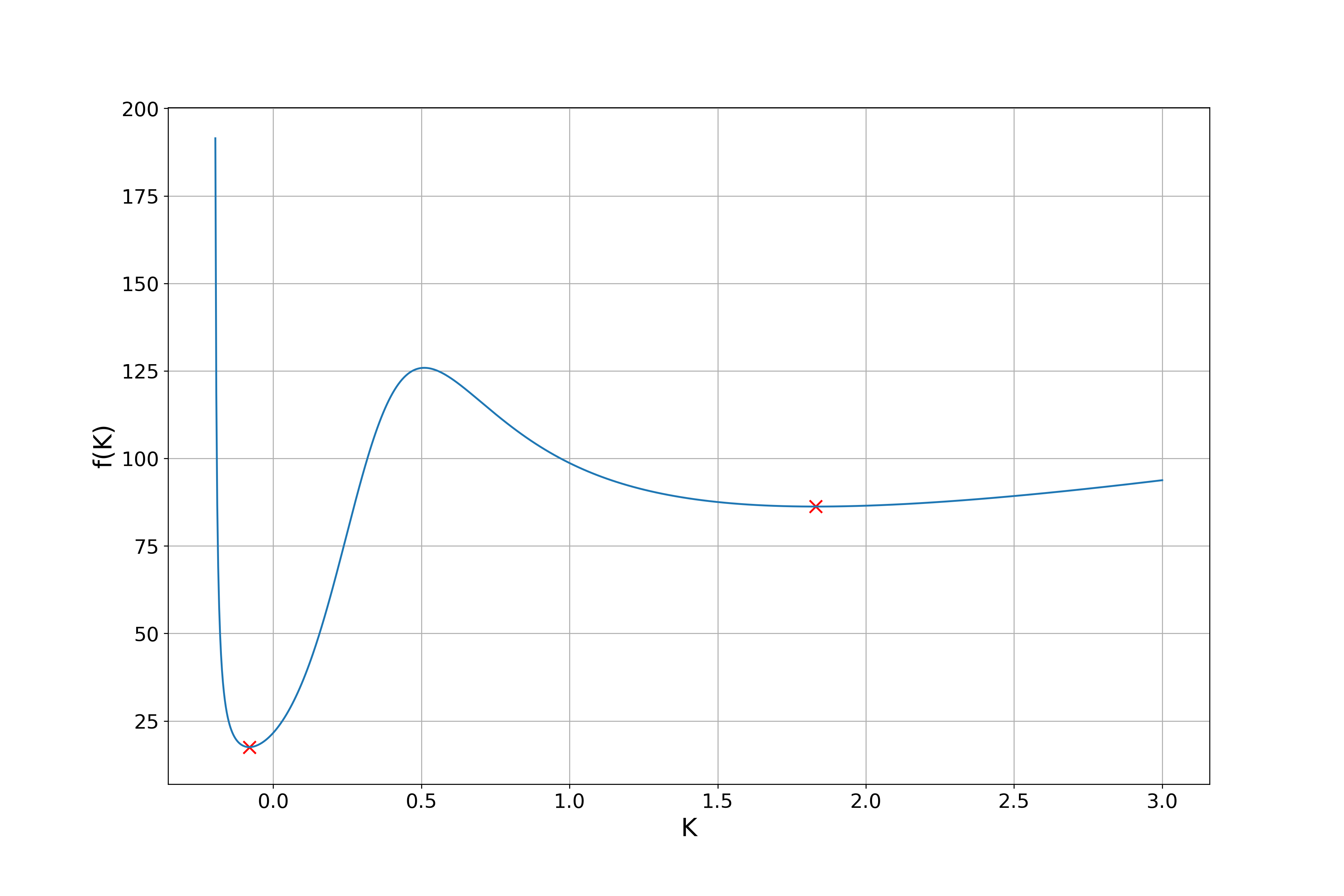}
		\caption{$f_{O}(K)$ for scalar output control with local minima in the same connected component.}
	\label{fig:1d_two_local}
	\end{minipage}
\end{figure}

 \begin{example}\label{example:saddle_point}
 \[
A = \begin{pmatrix} 0 & 1 & 0\\ 0 & 0 & 1\\-1 & -1& -\alpha \end{pmatrix}, B = \begin{pmatrix} 0 \\ 0 \\ 1 \end{pmatrix}, C = \begin{pmatrix} 1 & 1 & 0\\ 1 & -1 & 1 \end{pmatrix}, Q = I_3, R = 1.
\]
 \end{example}
The set
 \[
  \mathcal{S} = \left\{K \in \mathbb{R}^{1\times 2}:  \alpha+k_2 > 0, 1+k_1+k_2 > 0, (1+k_2)(\alpha+k_1-k_2)>1+k_1+k_2 \right\}
 \]
is connected with two local minima and a saddle point  \(K=(1.95, 0.38)\) for \(\alpha = 1.2\) (\cref{fig:two_d_saddle_0}). If \(\alpha\) is set to \(0.9\) there are two connectivity components with a single local minimum in each component (\cref{fig:two_d_two_comp_0}).
  \begin{figure}[!]
	\begin{minipage}{0.5\textwidth}
		\includegraphics[width=1\linewidth]{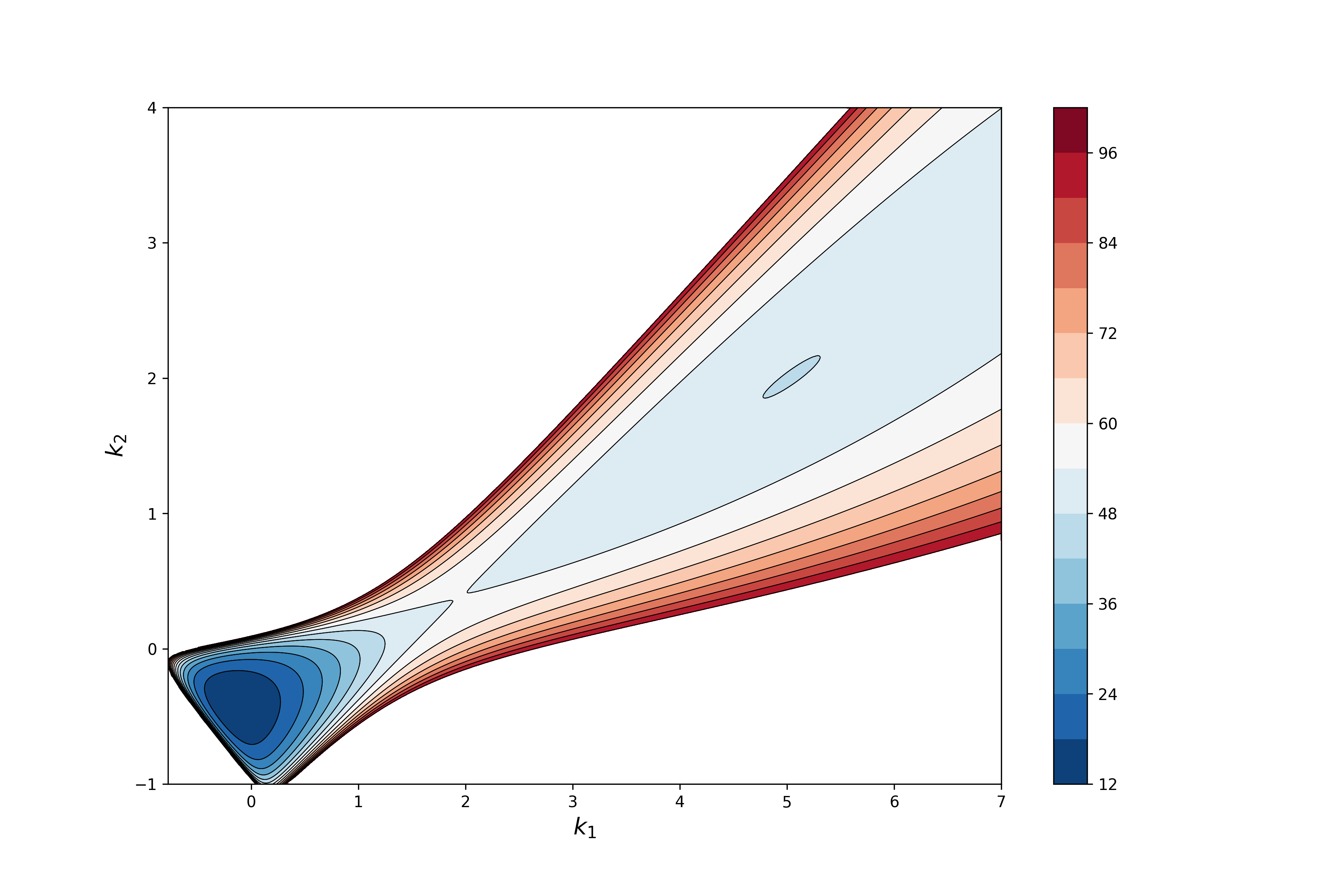}
		\caption{Two local minima of $f_{O}(K)$}
		\label{fig:two_d_saddle_0}
	\end{minipage}\hfill
	\begin{minipage}{0.5\textwidth}
		\centering
		\includegraphics[width=1\linewidth]{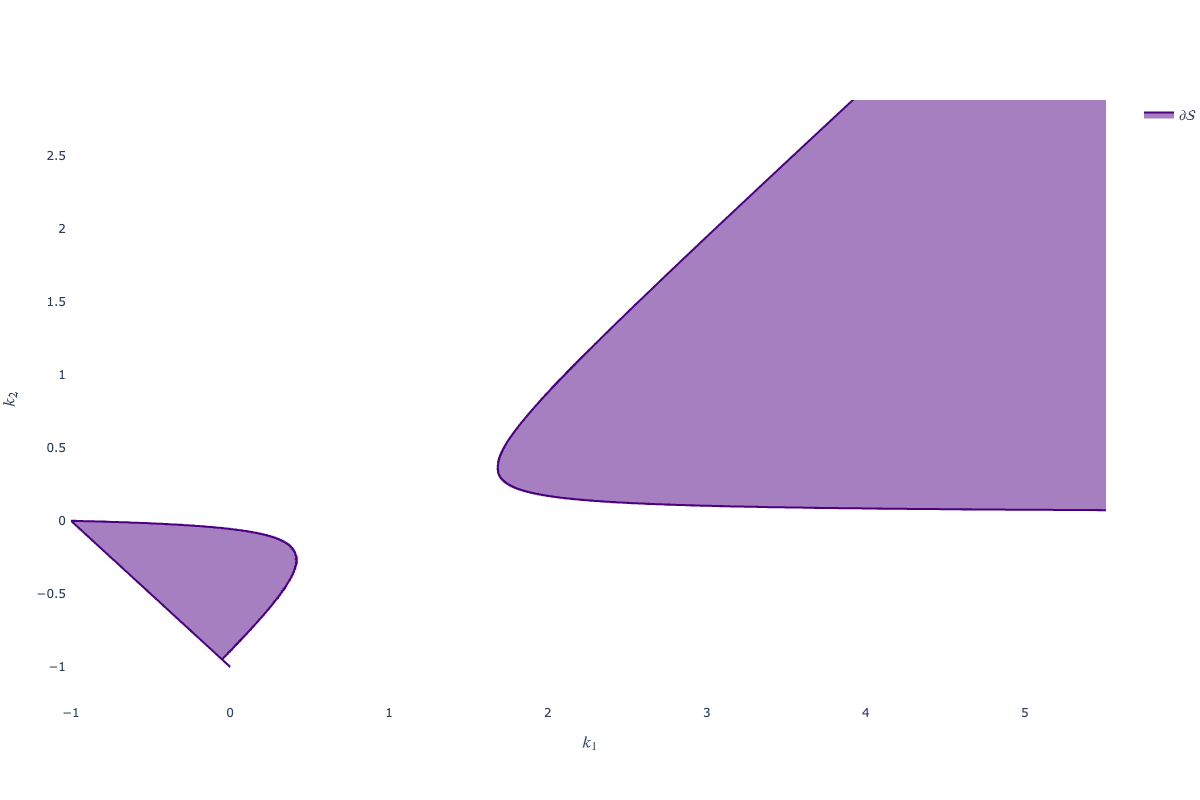}
		\caption{Two connectivity components of $ \mathcal{S}$ }
		\label{fig:two_d_two_comp_0}
	\end{minipage}
\end{figure}

We conclude that domain $ \mathcal{S}$ of $f(K)$ can be nonconvex with non-smooth boundary even for SLQR, and disconnected for OLQR. Function $f(K)$ can be unbounded on its domain but it looks smooth. We shall validate these properties below.

\subsection{Connectednes of  $\mathcal{S}$, $\mathcal{S}_0$}
It was known that $\mathcal{S}$ in state feedback case is connected \cite{LQR_continuous}, and the same is true for $\mathcal{S}_0$.
 \begin{lemma}\label{connectivity}
Let \(C = I\). The sets \(\mathcal{S}, \mathcal{S}_0\) are connected for every \(K_0\in \mathcal{S}\).
 \end{lemma}
 \begin{proof}
For  \(C = I\) equation (\ref{eq:Lyapunov_eq_X}) becomes \((A-B K )^{\top} X + X (A-B K )+ K^{\top} R K +Q=0\). It is proved in \cite{Khlebnikov_new_solution} that equality here can be replaced with inequality and after change of variables \(P=X^{-1} \) definition of stabilizing controllers becomes \[\mathcal{S}=\{K=R^{-1}B^T P^{-1},\, AP+PA^T-BR^{-1}B^T+PQP\preceq 0, P\succ 0\}.\]	
The inequality for \(P\) can be rewritten as block LMI and defines a convex set. Its image given by the continuous map \(K=R^{-1}B^T P^{-1}\) is connected. Similarly the set \(\mathcal{S}_0 \) is defined by the same map for the same set of \(P\) with extra constraint \(Tr P^{-1}\Sigma\le f(K_0)\) which is convex (again it can be written as LMI in \(P\)), this implies connectedness of \(\mathcal{S}_0 \). 
 \end{proof}
We provided the proof to demonstrate well known technique of variable change \cite{Boyd_LMI} which allows to transform the original problem  to a convex one. This line of research was developed in \cite{LQR_continuous,Mohammadi} to validate the gradient method. Unfortunately this trick does not work for output feedback --- there exist no convex reparametrization in this case.

As we have seen in Examples, the set $\mathcal{S}$ can be non-connected. Upper estimates for the  number $N$ of connected elements for particular cases may be found in \cite{Gryazina_2008}. For instance, if $m=r=1$ (single-input single-output system) then $N\le n+1$. For more general problems with additional condition $K\in L, L$ being a linear subspace in the set of matrices (so-called \textit{decentralised control}) the number of components can grow exponentially, see  \cite{Lavaei_exponential_2019}, where numerous examples can be found. 
 
\subsection{$f(K)$ is coercive and $\mathcal{S}_0$ is bounded}
The Examples exhibit that function $f(K)$ is unbounded on its domain. Below we analyse its behavior in more details.
\begin{definition}
A continuous function \(f: K \mapsto f(K) \in \mathbb{R}\) defined on the set \(\mathcal{S}\) is called \textit{coercive} if for any sequence \(\left\{K_{j}\right\}_{j=1}^{\infty} \subseteq \mathcal{S}\)
\[
f(K_{j}) \rightarrow +\infty\]
\[
\text { if }\|K_{j}\| \rightarrow + \infty \text{ or } K_{j} \rightarrow K \in \partial \mathcal{S}.
\]
\end{definition}
\begin{lemma}\label{lemma:coerciveness}
The function \(f(K) = \operatorname{Tr\left(X(K)\Sigma \right)}\) is coercive and the following estimates hold
\begin{equation}\label{lowerbound1}
f(K)\ge \frac{\lambda_{1}(\Sigma)\lambda_{1}(Q)}{-2 \Re{\lambda_n(A_K)}}, 
\end{equation}
\begin{equation}\label{lowerbound2}
f(K)\ge \frac{\lambda_{1}(\Sigma)\lambda_{1}(R)||K||^2_F\lambda_{1}(CC^T)}{2||A||+2||K||_F||B||||C||}.  
\end{equation}
\end{lemma}
The proof of the Lemma and further results can be found in \cref{appendix:output,appendix:state}. From estimate (\ref{lowerbound2}) we immediately get
\begin{corollary}
	For any \(K_0 \in \mathcal{S}\) the set \(\mathcal{S}_0\) is bounded.
\end{corollary}
On the other hand a minimum point of $f(K)$ on $\mathcal{S}_0$ exists (continuous function on a compact set) but $\mathcal{S}_0$ has no common points with boundary of $\mathcal{S}$ due to (\ref{lowerbound1}). Hence

\begin{corollary}
There exists a minimum point \(K_* \in \mathcal{S}\). 
\end{corollary}

This reasoning can be seen as an alternative proof of lemma 2.1 in \cite{Toivonen}.

\subsection{Gradient of $f(K)$}\label{sec:Properties_grad}
Differentiability of $f(K)$ is a well known fact, proved in the pioneering papers by Kalman \cite{Kalman} for SLQR and by Levine and Athans \cite{Levine_Athans} for OLQR. We provide it for completeness.
\begin{lemma}\label{lemma:LQR_grad_proof}
For all \(K \in \mathcal{S}\), the gradient of (\ref{eq:LQR_cost}) is 
\begin{equation}\label{LQR_gradient}
\nabla f(K)=2\left(R K C-B^{T} X\right) YC^{\top},
\end{equation}
where \(Y\) is the solution to the Lyapunov matrix equation
\begin{equation}\label{eq:Lyapunov_eq_Y}
A_{K} Y +Y A_{K}^{\top}+\Sigma=0.
\end{equation}
\end{lemma}
\begin{proof}
Consider the increment of the Lyapunov equation \cref{eq:Lyapunov_eq_X}
\[
A_K^{\top}dX + dX A_K + dA_K^{\top} X + X dA_K + C^{\top}dK^{\top} R K C + C^{\top}K^{\top} R dK C=0,
\]
\[
A_K^{\top}dX + dX A_K  +C^{\top}dK^{\top} (R K C - B^{\top}X) + (C^{\top}K^{\top} R  - X B) dK C=0.
\]
Denote \(M := R K C - B^{\top}X\) then
\[
df(K) = \operatorname{Tr}\left( \Sigma dX \right) = 2 \operatorname{Tr}\left( Y C^{\top} dK^{\top} M\right) = \langle 2M Y C^{\top}, dK \rangle,
\]
where \(Y\) is the solution to \cref{eq:Lyapunov_eq_Y}.
\end{proof}

 The necessary condition for the minimizer of \(f(K)\) is $\nabla f(K_*)=0$ (because $K_*$ exists and belongs to the open set $\mathcal{S}$). This condition implies the set of three nonlinear matrix equations for $K_*$: $\nabla f(K_*)=0$, (\ref{eq:Lyapunov_eq_Y}), (\ref{eq:Lyapunov_eq_X}). In general they can not be solved explicitly and numerical methods are required.

However there is the famous case of state feedback control \(C = I\) when explicit form of the solution (going back to Kalman \cite{Kalman}) can be obtained. Then  by setting the gradient calculated in \Cref{lemma:LQR_grad_proof} to zero and noting that \(Y_* \succ 0, C=I\) we get
\[
K_* = R^{-1}B^{T} X_*.
\]
Further, substituting the control matrix in \cref{eq:Lyapunov_eq_X} by the expression for \(K_*\) we obtain the well known Riccati equation for \(X_*\)
\[
A^{T} X_{*} - X_{*} B R^{-1} B^{T} X_{*} + X_{*} A - X_{*} B R^{-1} B^{T} X_{*} + X_{*} B R^{-1}R R^{-1} B^{T} X_{*} + Q=0,
\]
\[
A^{T} X_{*}+X_{*} A-X_{*} B R^{-1} B^{T} X_{*}+Q=0.
\]
Of course this is not completely explicit solution because Riccati equation should be solved numerically, but the methods for this purpose are well developed \cite{Hans_Ric,Penzl}. 

\subsection{Second derivative of $f(K)$}\label{sec:Properties_hess}
The performance index $f(K)$ is twice differentiable. To avoid tensors, we restrict analysis with
the action of the Hessian \( \nabla^2 f(K))[E, E]\) on a matrix \(E \in \mathbb{R}^{m\times n}\). It is given by the expression
\begin{equation}\label{Hessian_0}
\frac{1}{2}\nabla^{2} f(K)[E, E]=\left\langle\left(R E C-B^{\top}X^{\prime}(K)[E]\right) Y C^{\top}, E\right\rangle+ \left\langle M Y^{\prime}(K)[E] C^{\top}, E\right\rangle,
\end{equation}
where \(X^{\prime} := X^{\prime}(K)[E]\) and \(Y^{\prime} := Y^{\prime}(K)[E]\) are the solutions to equations
\[
A_{K}^{\top}X^{\prime} + X^{\prime}A_{K} + (-B E C)^{\top} X + X(-B E C)+ C^{\top} E^{\top} R K C + C^{\top} K^{\top} R E C = 0,
\]
\[
A_{K}Y^{\prime}+ Y^{\prime}A_{K}^{\top} + (-B E C) Y + Y(-B E C)^{\top}=0,
\]
which can be equivalently rewritten as 
\[
A_{K}^{\top}X^{\prime} + X^{\prime}A_{K} + M^{\top} E C + \left(M^{\top} E C\right)^{\top} = 0,
\]
\[
A_{K}Y^{\prime}+ Y^{\prime}A_{K}^{\top} -\left(B E C Y + (B E C Y)^{\top}\right)=0.
\]
Then applying \cref{basic_fact_1} and substituting \(X^{\prime}\) for \(Y^{\prime}\) in the last term of \cref{Hessian_0} we obtain
\begin{lemma}
For all \(K \in \mathcal{S}\), the gradient of \(f(\cdot)\) is differentiable and the action of the Hessian of \(f(\cdot)\) on any \(E \in \mathbb{R}^{m\times n}\) satisfies
\begin{equation}\label{Hessian_1}
\frac{1}{2}\nabla^{2} f(K)[E, E]=\left\langle R E C Y C^{\top}, E\right\rangle - 2\left\langle B^{\top}X^{\prime} Y C^{\top}, E\right\rangle.
\end{equation}
\end{lemma}

As Examples show, $f(K)$ is in general nonconvex. However for state feedback case we can guarantee local strong convexity in the neighborhood of the minimum point \(K_*\).
\begin{corollary}\label{corr:strong_conv_near_opt}
\(f_{S}(\cdot)\) is strongly convex in the neighborhood of \(K_*\).
\end{corollary}
\begin{proof}
Note that when \(K = K_*\) the second term in \cref{Hessian_1} turns to zero. If we recall that \(R, Y \succ 0\) it is straightforward to show that
\[
\left\langle R E Y, E\right\rangle = \operatorname{Tr}\left( (R^{\frac{1}{2}}E) Y (R^{\frac{1}{2}}E)^{\top}\right) > 0,
\]
Then the Hessian is positive definite at \(K_*\) and there is a neighbourhood of \(K_*\) where the function \(f_{S}(\cdot)\) is strongly convex.
\end{proof}
The upper bound for the second derivative is available.
\begin{lemma}\label{lemma:Hess_bound}
On the set \(\mathcal{S}\) the action of the Hessian \(\nabla ^2 f(K)\) on a matrix \(E \in \mathbb{R}^{m\times n}, \|E\|_F = 1\) can be bounded as
\begin{equation}
    \frac{1}{2} \nabla^{2} f(K)[E, E] \leq \lambda_n (R) \lambda_n(C Y C^{\top}) + \|X^{\prime}\|_F \|B\|\|C\|_F\|Y\|,
\end{equation}
where \(X^{\prime}\) and \(Y\) are solutions to the Lyapunov matrix equations 
\[
A_{K}^{\top}X^{\prime} + X^{\prime}A_{K} + M^{\top} E C + \left(M^{\top} E C\right)^{\top} = 0,
\]
\[
A_{K} Y +Y A_{K}^{\top}+\Sigma=0.
\]
\end{lemma}
\begin{proof}
It follows from \cref{Hessian_1} that
\[
\frac{1}{2}\sup _{\|E\|_{F}=1}\left|\nabla^{2} f(K)[E, E]\right| \leq \sup _{\|E\|_{F}=1}\left( \left|\left\langle R E C Y C^{\top}, E\right\rangle \right| + 2 \left|\left\langle B^{\top}X^{\prime} Y, E\right\rangle \right| \right)
\]
Now we estimate both terms in this expression separately assuming \(\|E\|_{F}=1\):
\[
\left\langle R E C Y C^{\top}, E\right\rangle = \operatorname{Tr} \left(R E C Y C^{\top} E^{\top}\right) \leq \lambda_n(R) \lambda_n(CYC^{\top}).
\]
By Cauchy - Schwarz inequality
\[
|\left\langle B^{\top}X^{\prime} Y C^{\top}, E\right\rangle| = |\left\langle X^{\prime}, B E C Y\right\rangle| \leq \|X^{\prime}\|_{F} \|B E C Y\|_{F}.
\]
It suffices to bound \( \|B E C Y\|_{F}\) when \(\|E\|_F = 1\)
\[
\|B E C Y\|_{F} = \sqrt{\operatorname{Tr}\left(BECYYC^{\top}E^{\top}B^{\top}\right)} \leq \|B\| \|C\|_{F} \|Y\|.
\]
\end{proof}
\subsection{$f(K)$ is L-smooth on $\mathcal{S}_0$}
A function is called \textit{L-smooth}, if its gradient satisfies Lipschitz condition with constant $L$.  Function $f(K)$ fails to be L-smooth on $\mathcal{S}$, however it has this property on sublevel set $\mathcal{S}_0$. 
\begin{theorem}\label{thm:Lipschitz_const}
On the set \(\mathcal{S}_0\) the function \(f(K)\) is \(L\)-smooth with  constant 
\begin{equation}\label{Lipschitz_const}
L = \frac{2f(K_0)}{\lambda_1(Q)}\left(\lambda_n(R)\|C\|^2 + \|B\|\|C\|_F \xi \right),
\end{equation}
where \(\xi = \frac{\sqrt{n} f(K_0)}{\lambda_1(\Sigma)}\left( \frac{f(K_0)\|B\|}{\lambda_1(\Sigma)\lambda_1(Q)} + \sqrt{\left(\frac{f(K_0)\|B\|}{\lambda_1(\Sigma)\lambda_1(Q)}\right)^2 + \lambda_n(R)}\right)\).
\end{theorem}
For the proof see \cref{appendix:output}.

\begin{corollary}
The following inequality holds for \(K\in \mathcal{S}_0\):
\begin{equation}\label{t_1}
\left|\nabla^{2} f(K)[E, E]\right| \leq L\|E\|_{F}^2
\end{equation}
where $L$ is given in (\ref{Lipschitz_const}).
\end{corollary}
Indeed for twice differentiable functions Lipschitz constant $L$ for gradients equals to the upper bound for the norm of second derivatives.

As we have seen for the examples, the boundary of \( \mathcal{S}\) can be non-smooth, while level sets of $f(K)$ are smooth due to \(L\)-smoothness property of \(f(K)\).

\subsection{Gradient domination property}
As we have seen, $f(K)$ can be noncovex even for state feedback case (SLQR). However there is a useful property which replaces convexity in validation of minimization methods. This property is referred to in the optimization literature as \textit{gradient domination} or \textit{Le\v{z}anski-Polyak-Lojasiewicz (LPL)} condition \cite{LPL_Polyak,Lezanski,Lojasiewicz,PL_Karimi}.

\begin{theorem}\label{thm:LPL}
 The function \(f_{S}(K)\) defined in \cref{eq:LQR_cost} satisfies the LPL condition on the set \(\mathcal{S}_0\)
\begin{equation}\label{PL}
\frac{1}{2} \|\nabla f_{S}(K)\|_{F}^{2} \geq \mu(f_{S}(K)-f_{S}\left(K_{*}\right))
\end{equation}
where \(\mu>0\) is given by
\begin{equation}\label{mu_const}
    \mu = \frac{\lambda_{1}\left(R\right)\lambda_1^2(\Sigma)\lambda_{1}\left(Q\right)}{ 8f(K_*)\left(\|A\| + \frac{\|B\|^2 f_{S}(K_0)}{\lambda_1(\Sigma)\lambda_1(R)}\right)^2 }.
\end{equation}
\end{theorem}

Constant $\mu$ in the LPL condition depends on $K_0$ and tends to zero when $f_{S}(K)$ tends to infinity. The condition is false for the entire set $\mathcal{S}$, as can be seen from \cref{example:1d_simple}. The condition cannot be applied for output feedback - for instance, in \cref{example:1d_two_local} there are two disconnected components with different values of minima. Moreover in \cref{example:saddle_point} there are two local minima in the connected domain.

\section{Methods}\label{sec:Methods}
Now we proceed to versions of gradient method for minimization of $f(K)$. This is not a standard task, because function $f(K)$ is defined not on the entire space of matrices, it is unbounded on its domain and can be nonconvex. However the properties of the function obtained in \cref{sec:Properties} allow to get convergence results. In all cases, the gradient methods behave monotonically. For SLQR  global convergence to the single minimum point with linear rate can be validated. For OLQR global convergence to a stationary point holds. In all versions of the method, the known stabilizing controller $K_0$ serves as the initial point.
\subsection{Continuous Method}
First we consider the gradient flow defined by the system of ordinary differential equations
\begin{equation}\label{flow_system}
\left\{\begin{array}{l}
{\dot{K}(t)=-\nabla f(K)}, \\
{K\left(0\right)=K_{0}\in \mathcal{S}}.
\end{array}\right.
\end{equation}

\begin{theorem}\label{thm:cont}
The solution of the above system \(K_{t} = K(t)\in \mathcal{S}_0\) exists for all $t\ge 0$, $f(K_t)$ is monotone decreasing and 
\begin{equation}\label{cont1}
\nabla f(K_t) \xrightarrow[t \to \infty]{} 0, \quad \min_{0\le t\le T}||\nabla f(K_t)||^2 \le \frac{f(K_0)}{T}.
\end{equation}
If $C=I$ then $K_t$ converges to the global minimum point $K_*$ exponentially:
\begin{equation}\label{linear_rate_K}
\left\|K_{t}-K_{\star}\right\|_{F} \leq \frac{\sqrt{2L(f(K_0)-f(K_*))}}{\mu}e^{-\mu t},
\end{equation}
where \(\mu\) and \(L\) are determined in \cref{thm:LPL,thm:Lipschitz_const}.
\end{theorem}
The main idea of the proof is the equality $\frac{d}{dt}f(K)=-||\nabla f(K)||^2$, the details are in \cref{appendix:methods}.

\subsection{Discrete Method}
Consider the gradient method in general form 
    \begin{equation}\label{GD_const}
    K_{j+1}=K_j -\gamma_j \nabla f(K_j).
    \end{equation}
The properties obtained in \cref{thm:LPL,thm:Lipschitz_const} allow to establish convergence guaranties for the above method.
\begin{theorem}\label{thm:grad}
For arbitrary \(0 < \gamma_j \leq \frac{2}{L} \) method (\ref{GD_const}) generates nonincreasing sequence \(f(K_j)\):
\begin{equation}\label{grad2}
f(K_{j+1})\le f(K_j)-\gamma_j\left(1-\frac{L\gamma_j}{2}\right) ||\nabla f(K_j)||_F^2. 
\end{equation}
Moreover if 
\(0<\varepsilon_1\le\gamma_j\le\frac{2}{L}-\varepsilon_2, \, \varepsilon_2>0 \) then
\[\nabla f(K_j) \rightarrow 0, \quad \min_{0\le j\le k}||\nabla f(K_j)||_F^2\le \frac{f(K_0)}{c_1 k}, \quad c_1=\frac{\varepsilon_1 \varepsilon_2 L }{2}. \]
and for \(C=I\) the method converges to the global minimum $K_*$ with a linear rate
\begin{equation}\label{grad3}
||K_j - K_*||\le cq^j, \, 0\le q<1.   
\end{equation}
\end{theorem}
The simplest choice is \(\gamma_j=1/L\), then in the last inequality
constants $ c, q$ can be written explicitly. The proof in \cref{appendix:methods} is mainly the  replica of the standard ones in \cite{LPL_Polyak}; however the non-trivial part is the proof that all iterations remain in $\mathcal{S}_0$. 

\subsection{Algorithm}
The method above is just a ``conceptual'' one, we do not know constant $L$ and it is hard to estimate it. Thus an implementable version of the algorithm is needed. It can be constructed as follows. Inequality (\ref{grad2}) provides the opportunity to apply Armijo-like rule: step-size $\gamma$ satisfies this rule if
\[
 f(K-\gamma \nabla f(K))\le f(K) -\alpha \gamma ||\nabla f(K)||_F^2
\]
for some $0<\alpha<1$. We can achieve this inequality by subsequent reduction of the initial guess for $\gamma$ due to (\ref{grad2}). This initial guess can be taken as follows.
 Consider a univariate function
\begin{equation}\label{varphi}
\varphi(t) = f(K-t \nabla f(K)),
\end{equation}
One iteration of Newton method for minimization of $\varphi(t)$ starting from $t_0=0$ implies
\[t_1= \frac{\varphi'(0)}{\varphi''(0)}.\]
Calculating derivatives we get
\begin{equation}\label{step_size_rule}
t_1 = \frac{\|\nabla f(K)\|_{F}^2}{\nabla^{2} f(K)[\nabla f(K), \nabla f(K)]}.
\end{equation}
But expressions for these quantities were obtained in \cref{sec:Properties} (see \cref{LQR_gradient,Hessian_1}). Notice that $t_1\ge 1/L$ due to (\ref{t_1}), thus such step-size is bounded below. Taking $\gamma_j=\min\{t_1, T_1\}$ with some $T_1>0$ (such upper bound is needed to restrict the step-size) for $K=K_j$ in gradient method we arrive to the basic algorithm below.

\begin{algorithm}
\caption{Gradient method}	
\label{alg:alg_1}
\begin{algorithmic}[1]
\STATE{Return: \(K\).}
\STATE{Initialization: \(K_0\in \mathcal{S}\), \(\epsilon>0, \alpha\in(0, 1), T_1 > 0\).}
 \WHILE{\(\|\nabla f(K)\|_F \geq \epsilon\)}
 \STATE{Solve for \(X\): \(A_K^{\top}X + XA_K + Q+K^{\top}RK = 0\).}
 \STATE{Solve for \(Y\): \(A_KY + Y A_K^{\top} + \Sigma = 0\).}
 \STATE{\(M \leftarrow RK-B^{\top}X\), \(\nabla f(K) \leftarrow 2 MY\).}
 \STATE{Solve for \(X'\): \(A_K^{\top}X' + X'A_K + M^{\top}\nabla f(K)+\nabla f(K)^{\top}M = 0\).}
 \STATE{\(\nabla^{2} f(K)[\nabla f(K), \nabla f(K)] \leftarrow 2\langle R\nabla f(K)Y, \nabla f(K)\rangle - 4\langle B^{\top} X'Y,\nabla f(K)\rangle\).}
 
 \STATE{\(t \leftarrow \min\{T_1, \frac{\|\nabla f(K)\|_{F}^2}{\nabla^{2} f(K)[\nabla f(K), \nabla f(K)]}\}, K_{prev} \leftarrow K\).}
 \STATE{Gradient step: \(K \leftarrow K - t \nabla f(K)\).}
 \IF{\(K \in \mathcal{S}\) or \(f(K) \geq f(K_{prev}) - \alpha t\|\nabla f(K_{prev})\|_{F}^2\)}
   \STATE{\(t \leftarrow \alpha t\),}
    \STATE{repeat the gradient step.}
  \ENDIF
 \ENDWHILE
\end{algorithmic}
\end{algorithm}
\begin{theorem}\label{thm:alg}
	For \cref{alg:alg_1} the number of step reductions is bounded uniformly for all iterations and
	convergence results of \cref{thm:grad} hold true.
\end{theorem}
The proof follows the same lines as for \cref{thm:grad} and is given in the Appendix. 

There are different ways to choose constants $T_1,\alpha$ in the Algorithm. We do not discuss them here, because there are various implementations of the Algorithm and they deserve a separate consideration. 

It is also possible to consider a different approach for a stepsize choice. For instance, it can be chosen in such a way that guaranties that a new iterate remains stabilizing. Then there is no need to check if \(K\in \mathcal{S}\) on every iteration. Consider the Lyapunov equation
\[
(A-BKC)Y + (A-BKC)^{\top} + I = 0.
\]
Denote \(K_t = K - t \nabla f(K)\) and \(G = (B\nabla f(K)C) Y + Y (B\nabla f(K)C)^{\top}\).
\[
AY+YA^{\top} - \left[ (BK_tC)Y + Y(BK_tC)^{\top}\right] + I - tG = 0,
\]
\[
A_{K_t}Y+YA_{K_t}^{\top} + I - tG = 0.
\]
The function \(V(x) = x^{\top}Y^{-1}x\) remains the quadratic Lyapunov function for a new \(A_{K_t}\) when \(I-tG \succ 0\). If \(\lambda_{max}(G) \leq 0\), then \(K_t \in \mathcal{S}, \forall t > 0\). Otherwise, \(K_t \in \mathcal{S}\) if \(0 < t < \frac{1}{\lambda_{max}(G)}\).

\section{Simulation}\label{sec:Simulation}
We have started with the comparison of various versions of the step-size choice of gradient descent method for low-dimensional tests, such as \cref{example:1d_simple,example:2,example:3,example:1d_two_local,example:saddle_point}. In all cases, \cref{alg:alg_1} was superior and converged to global or local minimizers with high accuracy in 10--20 iterations. 
 
For medium-size simulation we generated matrices with dimensions \(n = 100, m = 10\) for SLQR problem:
\[
C = I, \quad A = \frac{1}{n}rand(n, n) - I, \quad B = ones(n, m) + \frac{1}{2}rand(n, m),
\]
\[
 Q = Q_1 Q_1^{\top}, \quad Q_1 = rand(n, n), \quad R = R_1 R_1^{\top}, \quad R_1 = rand(m, m),
\]
where \(ones(n, m)\) is a \(n\times m\) matrix with all entries equal to one and \(rand(n, m)\) is a \(n\times m\) matrix with every entry generated from the uniform distribution between \(0\) and \(1\). We choose the initial stabilizing controller as $K_0 = 0$. It is indeed stabilizing because $A$ is Hurwitz. We find optimal gain $K_*$ by solving ARE, thus we could compare the accuracy of the obtained solutions. Then we apply three different versions of the first order methods to solve this problem. The first one is the simplest version of the gradient method  \cref{GD_const}   with constant step-size $\gamma_j=\gamma$ tuned at initial iterations to guarantee monotonicity of $f(K_j)$, it is denoted as $GD\_r$. The second is our basic  \cref{alg:alg_1}  $(GDN)$ . The last one  is the conjugate gradient method described below by update rules in \cref{CG} $(CGN)$.  The convergence of the methods is illustrated in \cref{fig:large_scale_C_I}. Of course, the simplest form of gradient method $GD\_r$ is very slow, because step-size should be strongly enlarged after initial iterations. Our basic algorithm $(GDN)$ converges satisfactory; it is worth mentioning that the number of step reductions or truncations is minimal (approximately 10 for 100 iterations), thus step-size rule \cref{step_size_rule} works with minor corrections at all stages of iteration process. Finally, the proposed version of the conjugate gradient method strongly accelerates convergence.

Of course these calculations are preliminary, much more should be done to develop reliable and efficient gradient-based algorithms for state feedback which can win in competition with classical algorithms based on Riccati-equation techniques. The behavior of the method for output feedback also requires detailed investigations.

\begin{figure}[!]
		\includegraphics[width=1\linewidth]{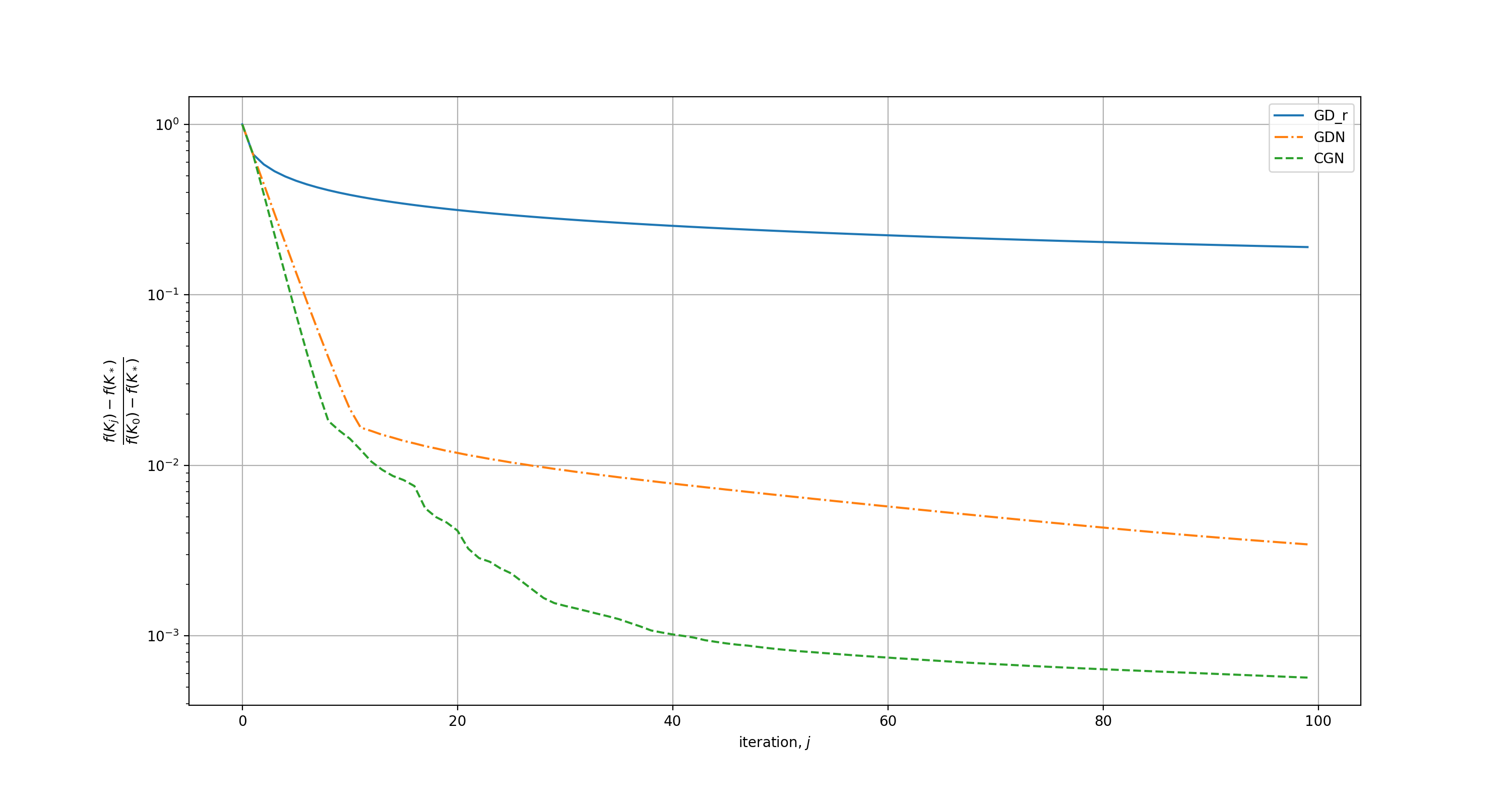}
		\caption{Three first-order methods for $n=100, m=10$}
		\label{fig:large_scale_C_I}
\end{figure}

\section{Links with general optimization problems}\label{sec:Future_work}
The results obtained above for the particular feedback minimization problem can provide some surplus profit for the analysis of several abstract formulations for unconstrained and constrained optimization. We consider three such "side effects".

\subsection{Step-size choice for gradient descent}
The step-size rule proposed in \cref{step_size_rule} is also valid for a general setup of smooth unconstrained optimization problem
\[
\min_{x\in \mathbb{R}^n} f(x).
\]
The gradient method becomes
\begin{equation}\label{step_general}
x_{j+1}=x_j-\gamma_j \nabla f(x_j), \quad \gamma_j = \frac{\|\nabla f(x_j)\|^2}{\langle \nabla^2 f(x_j) \nabla f(x_j), \nabla f(x_j)\rangle}
\end{equation}
and it is promising whenever the problem structure allows an efficient computation of the quadratic form in the denominator. It is particularly attractive in practice, because it does not require the knowledge of constants $L$ and $\mu$ and uses a second order information at a minor cost. For quadratic functions $f(x)=(Hx,x)$ the method coincides with the steepest descent. For nonquadratic functions its rigorous validation  is possible for strongly convex case. 

\begin{theorem}\label{thm:general_local}
	Let \(f(\cdot)\) be twice differentiable $\mu$-strongly convex function in \(\mathbb{R}^n\), \(\nabla f(\cdot)\) and \(\nabla^2 f(\cdot)\) Lipschitz continuous with constants \(L\) and \(M\) respectively. Then if  the initial condition \(x_0\) satisfies
	\begin{equation}\label{local_condition}
	M \sqrt{2L(f(x_0) - f(x_*))} \leq 3\mu^2\left(1-\delta \right), \quad \delta > 0, 
	\end{equation}
	then the method  \cref{step_general} converges to the global minimizer $x_*$ with a linear rate:
	\begin{equation}
	f(x_{j}) - f(x_*) \leq \left(f(x_{0}) - f(x_*)\right) \left( 1 - \frac{\mu \delta}{L}\right)^j.
	\end{equation}

The damped version of (\ref{step_general})  ($\gamma_j$ replaced with $\sigma \gamma_j, \sigma\le \frac{\mu}{L}$) converges  for an arbitrary $x_0$:
	\begin{equation}
	f\left(x_{j}\right)-f\left(x_{*}\right) \leq\left(f\left(x_{0}\right)-f\left(x_{*}\right)\right)\left(1-\frac{\mu\sigma}{ L}\right)^{j}.
	\end{equation}
\end{theorem}
The proofs are deferred to \cref{appendix:methods}.

Similar step-size rule can be applied for the solution of constrained minimization problem $\min_{Q} f(x)$ via  \textit{gradient projection} method \cite{Intro_Opt_Polyak}.

\subsection{New version of the conjugate gradient method}
Similar approach can be exploited for the conjugate gradient method for unconstrained minimization of $f(x)$
in $\mathbb{R}^n$. The standard version of the method requires 1D minimization for finding step-size $\alpha_j$, but it can be replaced as follows:
\begin{equation}\label{CG}
x_{j+1}=x_j+\alpha_j p_j, \quad \alpha_j = \frac{\|p_j\|^2}{( \nabla^2 f(x_j) p_j, p_j)},
\end{equation}
$$p_j=-\nabla f(x_j)+\beta_j p_{j-1}, \quad \beta_j=\frac{\|\nabla f(x_j)\|^2}{\|\nabla f(x_{j-1})\|^2}, \quad \beta_0=0.$$
There are various formulae for $\beta_j$, see e.g. \cite{Intro_Opt_Polyak}, we provided above just the simplest one. Probably, convergence results for (\ref{CG}) can be obtained.
 
\subsection{Reduced gradient method}

Gradient method for feedback minimization can be considered  in general setup of abstract optimization problem with equality-type constraints
\begin{align*}
&\min_{x, y} f(x,y), \\&\text{s.t. } g(x,y)=0,
\end{align*}
here $ x\in \mathbb{R}^n, y\in \mathbb{R}^m, g(x)\in \mathbb{R}^n. $ Suppose that  the solution $x(y)$ of the equality $g(x,y)=0$ for fixed \(y \in \mathcal{S}\) can be found either explicitly or with minor computational efforts. Define \(F(y):= f(x(y),y)\). Thus problem is  converted to unconstrained optimization

$$\min_{y\in \mathcal{S}} F(y).$$

Gradient of $F(y)$ can be written with no problems
\begin{equation*}
	\nabla F(y)=-	\nabla_y g(x,y)^{\top} ((\nabla_x g(x,y))^{-1} )^{\top} 	\nabla_xf(x,y)^{\top}+	\nabla_y f(x,y)^{\top}
\end{equation*}
and gradient method with $y_0\in\mathcal{S}$ becomes so-called \textit{reduced gradient} method:
\begin{equation}\label{RG}
y_{j+1}=y_j-\gamma_j \nabla F(y_j), \quad x_j=x(y_j).
\end{equation}
The method has been proposed by Ph.Wolfe \cite{Wolfe} and implemented in numerous algorithms, see e.g. \cite{Abadie}. The standard  assumption was $\mathcal{S}=\mathbb{R}^n$. However the method for nonlinear equalty constraints had just local theoretical validation (see e.g. Theorem 8, Chapter 8.2 in \cite{Intro_Opt_Polyak}), while the main interest is its global convergence. In the setup of the present paper $x$ corresponds to $Y$, $y$ to $K$. The main tool for proving convergence in general case is to obtain the conditions which are the analogs of our results on $L$-smoothness and LPL-condition (\cref{thm:Lipschitz_const,thm:LPL}). If such results hold, the proof is a replica of our considerations.  

\section{Conclusion}\label{sec:Conclusion}
The results can be extended in several directions. First, more efficient computational schemes are of interest. Gradient method is the simplest method for unconstrained smooth optimization. Accelerated algorithms - such as conjugate gradient, heavy ball, Nesterov acceleration - are developed for strongly convex functions. But we have proved (\cref{corr:strong_conv_near_opt}) that $f_S (K)$ is strongly convex in the neighborhood of the optimal solution $K_*$. Thus such methods are applicable to accelerate local convergence. Second, more research should be devoted to output minimization. For instance, how common is the effect of multiple minima in one connectivity component (as in \cref{example:saddle_point})? Does the method converge to a local minima only or it can be a saddle point? Third, there is highly important research direction which unites the problems of control, optimization and machine learning and uses such approaches as \textit{policy optimization, reinforcement learning, adaptive control}, see the survey \cite{Hu_2019}. The gradient method can be easily extended to decentralised control (this additional condition $K\in L, L$ being a linear subspace in the space of matrices), see e.g. \cite{Lavaei_exponential_2019} and to other \textit{parametric LQR} problems \cite{Makila_Toivonen_Survey}. However its validation remains open question.   

\appendix\label{sec:Appendix}
\section{Basic Facts}\label{appendix:basic}
The following lemmas are helpful throughout the paper.
\begin{lemma}\label{basic_fact_1}
	Let \(X\) and \(Y\) be the solutions to the dual Lyapunov equations with Hurwitz matrix \(A\)
	\[
	A^{\top} X + X A + W = 0,
	\]
	\[
	A Y+Y A^{\top} + V = 0.
	\]
	Then \(\operatorname{Tr}\left(XV\right) = \operatorname{Tr}\left(YW\right)\).
\end{lemma}
\begin{lemma}\label{basic_fact_2}
	Let \(W_1 \succ W_2\) and \(X_1, X_2\) be the solutions to Lyapunov equations with Hurwitz matrix \(A\)
	\[
	A^{\top} X_1 + X_1 A + W_1 = 0,
	\]
	\[
	A^{\top} X_2 + X_2 A + W_2 = 0.
	\]
	Then \(X_1 \succ X_2\).
\end{lemma}
\begin{lemma}\label{basic_fact_3}
	Let \(N, L \in \mathbb{R}^{m\times n}\). Then for any \(\alpha > 0\)
	\[
	N^{\top}L + L^{\top}N \preceq \alpha N^{\top}N + \frac{1}{\alpha} L^{\top}L.
	\]
\end{lemma}
\begin{lemma}\label{basic_fact_4}
	For all positive semi-definite \(A, B \in \mathbb{R}^{n\times n}\), it holds that
	\[
	\lambda_1(A) \operatorname{Tr}(B) \leq \operatorname{Tr}(AB) \leq \lambda_n(A) \operatorname{Tr}(B).
	\]
\end{lemma}
\begin{lemma}\label{basic_fact_5}
If $X\succ 0$ is the solution of Lyapunov equation
\[
A^{\top} X + X A - Q = 0,
\]	
with $A$ Hurwitz and $Q\succ 0$ then
\[
\lambda_{n} (X) \ge\frac{\lambda_{1}(Q)}{2\sigma (A)}, \quad
\lambda_{1} (X) \ge\frac{\lambda_1(Q) }{2\|A\|},
\]	
\end{lemma}	
where \(\sigma(A) := -\max_{i}\left[ \Re\lambda_i(A)\right]\) is the \textit{stability degree} of $A$. These are well known lower bounds for Lyapunov equation, see e.g. \cite{Chien_Hua_Lee}.

\section{Analysis of the OLQR}\label{appendix:output}
\subsection{Proof of \cref{lemma:coerciveness}}

\begin{proof} Let us first consider the sequence \(\left\{K_{j}\right\}_{j=1}^{\infty} \subseteq \mathcal{S}\): \(K_{j} \rightarrow K \in \partial \mathcal{S},\) i.e.  \(\sigma(K) =0\). The stability degree  is a continuous map, i.e. \(\sigma\left(A-B K_{j}C\right) \rightarrow \sigma(A-B KC) .\) Therefore, \( \forall \varepsilon>0, \exists N=N(\varepsilon) \in \mathbb{N}\) such that
\[
\left|\sigma\left(A-B K_{j}C\right)-\sigma(A-B KC)\right| = \sigma\left(A-B K_{j}C\right) <\varepsilon, \quad \forall j \geq N .
\]
Let \(X_{j}\) be the solution to the corresponding Lyapunov equation \cref{eq:Lyapunov_eq_X} associated with \(K_{j} \), then 

\begin{equation}\label{eq:lower_bound_wo_int}
	f\left(K_{j}\right)=\operatorname{Tr}(X_{j} \Sigma) \geq \lambda_1(\Sigma) \operatorname{Tr}(X_{j}) \geq \frac{\lambda_1(\Sigma) \lambda_1(Q+C^{\top}K_{j}^{\top}RK_{j}C)}{2\sigma(A_{K_j})} \geq \frac{\lambda_1(\Sigma) \lambda_1(Q)}{2\sigma(A_{K_j})}  \geq
	\frac{\lambda_{1}(\Sigma) \lambda_{1}(Q)}{2\varepsilon} \rightarrow +\infty
\end{equation}
if $ \varepsilon \rightarrow 0$. Here the second inequality is based on \cref{basic_fact_5}.

On the other hand, suppose that the sequence \(\left\{K_{j}\right\}_{j=1}^{\infty} \subseteq \mathcal{S}: \|K_j\|\rightarrow +\infty\).
\begin{align*}
	f(K_{j}) &= \operatorname{Tr}\left(X_j \Sigma\right) = \operatorname{Tr}\left( Y_j (Q+C^{\top}K_{j}^{\top}RK_{j}C) \right) \geq
	\operatorname{Tr}\left( Y_j C^{\top}K_{j}^{\top}RK_{j}C \right) \geq 
	\lambda_{1}\left( Y_j \right) \operatorname{Tr}\left(C^{\top}K_{j}^{\top} R K_{j}C\right) \\&
	\geq \lambda_{1}(Y_j) \lambda_1(CC^{\top}) \operatorname{Tr}\left(K_{j}^{\top} R K_{j}\right)
	\geq \lambda_{1}(Y_j) \lambda_1(CC^{\top}) \lambda_1(R) \operatorname{Tr}\left(K_{j}^{\top} K_{j}\right)
	\geq \lambda_{1}(Y_j) \lambda_1(CC^{\top}) \lambda_{1}(R)\|K_{j}\|_{F}^{2} ,
\end{align*}
where \(Y_j\) is the solution to the Lyapunov equation \(A_{K_{j}} Y_j + Y_j A_{K_{j}}^{\top} + \Sigma=0.\) Here the second equality follows from \cref{basic_fact_1}, the first inequality is due to $Y_j, Q \succ 0$ and all the rest inequalities simply use \cref{basic_fact_4}.

But \cref{basic_fact_5} implies
\begin{equation}\label{lower_bound_Y}
\lambda_1(Y_j) \geq \frac{\lambda_1(\Sigma)}{2\|A_{K_j}\|}.
\end{equation}
Therefore, we can complete the proof by estimating the denominator using properties of the matrix norm:
\[
\|A - B K_j C\| \leq \|A\| + \|B K_j C\| \leq \|A\| + \|B K_j C\|_F \leq \|A\| + \sqrt{\operatorname{Tr}(C^{\top}K_j^{\top}B^{\top}B K_j C)} \leq \|A\| + \|K_j\|_F \|B\| \|C\|,
\]
where the last inequality uses \cref{basic_fact_4} and the definition $\|C\| = \sqrt{\lambda_n(C C^{\top})}$, $\|B\| = \sqrt{\lambda_n(B B^{\top})}$. Hence
\begin{equation}\label{bound_f_and_K}
f(K_j)\geq  \frac{\lambda_1(\Sigma)\lambda_1(CC^{\top})\lambda_{1}(R)\|K_{j}\|_{F}^{2} }{2\|A - B K_j C\|} \geq \frac{\lambda_{1}(\Sigma) \lambda_1(CC^{\top}) \lambda_{1}(R)\|K_{j}\|_{F}^{2} }{ 2\|A\| + 2 \|K_{j}\|_F \|B\|\|C\|} \rightarrow + \infty,
\end{equation}
if \(\|K_{j}\|_{F}\rightarrow +\infty\).
\end{proof}
\subsection{Proof of \cref{thm:Lipschitz_const}}
\begin{proof}
    Note that in \cref{lemma:Hess_bound} \(Y\) depends on \(K\) and \(X^{\prime}\) on \(K\) and \(E\). We should obtain a uniform estimate that depends only on the problem parameters and \(K_0\).
	The first term in \cref{lemma:Hess_bound} can be upper bounded as 
	\begin{equation}\label{first_in_tri}
	\lambda_n(R) \lambda_n(CYC^{\top}) \leq \frac{\lambda_n(R) }{\lambda_1(Q)}f(K_0)\|C\|^2.
	\end{equation}
	The above estimate follows from \cref{lemma:upper_bound_Y} and inequalities
	\[\lambda_n(CYC^{\top}) \leq \operatorname{Tr}(CYC^{\top}) \leq \operatorname{Tr}(Y) \|C\|^2.
	\]
	In view of $\|Y\| \leq \operatorname{Tr}(Y)$, the second term in \cref{lemma:Hess_bound} is bounded as
	\begin{equation}\label{second_in_tri_1}
	\|B\| \|C\|_{F} \|Y\| \leq \frac{\|B\|\|C\|_{F}}{\lambda_1(Q)} f(K_0).
	\end{equation}
	Further it suffices to bound \(\|X^{\prime}\|_{F} \). We first show that \(X^{\prime} \preceq \alpha X\) with some constant \(\alpha\). Recall that \(X^{\prime}\) is the solution to
	\[
	A_{K}^{\top}X^{\prime} + X^{\prime}A_{K} + C^{\top} K^{\top}RE C + C^{\top}E^{\top}RK C - (XBEC +(XBEC)^{\top})=0.
	\]
	By utilizing the facts from \cref{appendix:basic} we obtain for any \(\alpha, \beta > 0\) that \(X^{\prime} \preceq \tilde{X}^{\prime}\), where \(\tilde{X}^{\prime}\) is the solution to 
	\[
	A_{K}^{\top}\tilde{X}^{\prime} + \tilde{X}^{\prime}A_{K} + \alpha C^{\top}K^{\top}RKC + \frac{1}{\alpha} C^{\top} E^{\top}REC + \left(\beta X^2 +\frac{1}{\beta}(BEC)^{\top}BEC \right) = 0.
	\]
	Further we divide the previous equation by $\alpha >0$ and aim to choose the constants \(\alpha\) and \(\beta\) to ensure that \(\tilde{X}^{\prime} \preceq \alpha X\).
	\begin{equation}\label{Lyap_divided}
	A_{K}^{\top}\left(\frac{\tilde{X}^{\prime}}{\alpha}\right) + \left(\frac{\tilde{X}^{\prime}}{\alpha}\right)A_{K} + C^{\top} K^{\top}RKC + \frac{1}{\alpha^2} C^{\top}E^{\top}REC + \frac{1}{\alpha}\left(\beta X^2 +\frac{1}{\beta}(BEC)^{\top}BEC \right) = 0.
	\end{equation}
	Consider the matrix function of two variables 
	\[
	F(\alpha, \beta) := C^{\top}E^{\top}\left(\frac{1}{\alpha}R+\frac{1}{\beta}B^{\top}B\right)EC + \beta X^{2} - \alpha Q.
	\]
	To obtain an upper bound on \(\tilde{X}^{\prime}\) we solve the two dimensional minimization problem on \((\alpha, \beta)\) with the relaxed matrix inequality constraint \(F_1(\alpha, \beta) \preceq F(\alpha, \beta) \preceq 0\):
	\begin{align*}
		&\alpha \rightarrow \min_{\alpha, \beta > 0 }, \\
		&F_1(\alpha, \beta) \preceq 0,
	\end{align*}
	where
	\[
	F_1(\alpha, \beta) := \left(\frac{1}{\alpha}\lambda_n(R)+\frac{1}{\beta}\|B\|^2 + \beta \|X\|^2 - \alpha \lambda_1(Q)\right)I.
	\]
	The solution is the pair
	\[
	\alpha_* = \frac{\|X\|\|B\| + \sqrt{\|X\|^2\|B\|^2 + \lambda_1(Q)\lambda_n(R)}}{\lambda_1(Q)} ,
	\beta_* = \frac{\|B||}{\|X\|}.
	\]
	Note that it trivially follows from \(X \preceq \frac{f(K_0)}{\lambda_1(\Sigma)}I\) that \(X^2 \preceq \frac{f^2(K_0)}{\lambda_1^2(\Sigma)}I\).
	Therefore,
	\[
	X^{\prime} \leq \alpha_* X \leq \frac{\alpha_*}{\lambda_1(\Sigma)}  f(K_0) I. 
	\]
	Let us denote \(\eta := \frac{\alpha_*}{\lambda_1(\Sigma)} f(K_0)\) and note that as \(X^{\prime}\) and \(\eta I\) commute we obtain the bound on the Frobenius norm 
	\begin{equation}\label{second_in_tri_2}
	\|X^{\prime}\|_F \leq \sqrt{n} \eta \leq \frac{\sqrt{n}f(K_0)}{\lambda_1(\Sigma)}\left( \frac{f(K_0)\|B\|}{\lambda_1(\Sigma)\lambda_1(Q)} + \sqrt{\left(\frac{f(K_0)\|B\|}{\lambda_1(\Sigma)\lambda_1(Q)}\right)^2 + \lambda_n(R)}\right) := \xi.
	\end{equation}
	The result \cref{Lipschitz_const} follows directly from \cref{lemma:Hess_bound} if we apply the obtained bounds \cref{first_in_tri,second_in_tri_1,second_in_tri_2}.
\end{proof}
\section{Analysis of SLQR}\label{appendix:state}
\subsection{Technical Lemmas}
\begin{lemma} \label{lemma:almost_LPL}
	Consider the state feedback control (i.e. \(C = I\) ). Let \(K_{*} \in \mathcal{S}\) be the optimal feedback gain and \(K_0 \in \mathcal{S}\). Then for \(K \in \mathcal{S}_0\)
	\begin{equation}\label{eq:almost_LPL}
	f_{S}(K)-f_{S}(K_*) \leq  \frac{(\|A\|+ \|K\|_F\|B\|)^2\lambda_n(Y_*)}{\lambda_1(R) \lambda_1(\Sigma)} \|\nabla f_{S}(K)\|_F^2,
    \end{equation}

	where 
	\(Y_{*}\) is the solution to the Lyapunov matrix equation
	\begin{equation}\label{Lyap_Y_star}
	A_{K_{*}} Y_{*} +Y_{*} A_{K_{*}}^{\top}+\Sigma=0.
	\end{equation}
\end{lemma}
\begin{proof}
	The Lyapunov equations for arbitrary \(X = X(K)\) and \(X_{*} = X(K_{*})\) are
	\begin{equation}\label{Lyp_arb}
	A_{K}^{\top} X +X A_{K}+K^{\top} R K+Q=0,
	\end{equation}
	\begin{equation}\label{Lyp_star}
	A_{K_{*}}^{\top} X_{*} +X_{*} A_{K_{*}}+K_{*}^{\top} R K_{*}+Q=0.
	\end{equation}
	Substituting \cref{Lyp_arb} from \cref{Lyp_star} gives
	\begin{equation}\label{Lyp_dif}
	A_{K}^{\top} X - A_{K_{\star}}^{\top} X_{*} + X A_{K} - X_{*} A_{K_{*}}+K^{\top} R K - K_{*}^{\top} R K_{*}=0,
	\end{equation}
	which is equivalent to
	\begin{equation}\label{Lyp_equiv}
	A_{K_{\star}}^{\top} (X-X_{*}) + (X-X_{*}) A_{K_{*}} + (K-K_{*})^{\top}M + M^{\top}(K-K_{*}) - (K-K_{*})^{\top} R (K-K_{*})=0,
	\end{equation}
	where \(M = RK - B^{\top}X\).
	
	For any \(\alpha > 0\)
	
	\(\left(K-K_{\star}\right)^{\top}M+M^{\top}\left(K-K_{\star}\right) \leq \frac{1}{\alpha}\left(K-K_{\star}\right)^{\top}\left(K-K_{\star}\right)+\alpha M^{\top} M.\)
	
	Therefore, picking \(\alpha = \frac{1}{\lambda_{1}(R)}\) we obtain
	
	\begin{align*}
	&\left(K-K_{\star}\right)^{\top}M+M^{\top}\left(K-K_{\star}\right) - (K-K_{*})^{\top} R (K-K_{*}) \\&\leq \alpha M^{\top} M + (K-K_{*})^{\top} (\frac{1}{\alpha}I - R) (K-K_{*}) \\&\leq \frac{1}{\lambda_{1}(R)} M^{\top} M.
	\end{align*}
	
	Let \(Z\) be the solution to
	\[
	A_{K_{*}}^{\top} Z +Z A_{K_{*}}+\frac{1}{\lambda_{1}(R)} M^{\top}M=0.
	\]
	Then \((X-X_*) \leq Z\). Further,
	\begin{align*}
		&f_{S}(K)-f_{S}(K_*) = \operatorname{Tr}((X-X_\star) \Sigma) \leq
		\operatorname{Tr}(Z \Sigma) =
		\frac{1}{\lambda_{1}(R)}
		\operatorname{Tr} \left( M^{\top}M Y_* \right) \leq \frac{ \lambda_n(Y_*)}{\lambda_1(R)}\operatorname{Tr} \left( M^{\top}M \right) \\&
		\leq \frac{ \lambda_n(Y_*)}{\lambda_1(R) \lambda_1^2(Y)}\operatorname{Tr} \left( Y^{\top}M^{\top}M Y\right) = \frac{\lambda_n(Y_*)}{4\lambda_1(R) \lambda_1^2(Y)} \|\nabla f_{S}(K)\|_F^2,
	\end{align*}
	where \(Y\) satisfies
	\[
	A_{K}Y + YA_{K}^{\top} + \Sigma = 0.
	\]
	It follows from \cref{lower_bound_Y} that
	\[
	\lambda_1(Y) \geq \frac{\lambda_1(\Sigma)}{2\|A_K\|} \geq \frac{\lambda_1(\Sigma)}{2\left(\|A\|+ \|B\|\|K\|_F\right)} > 0.
	\]
	Therefore,
	\[f_{S}(K)-f_{S}(K_*) \leq  \frac{\left(\|A\|+ \|B\|\|K\|_F\right)^2\lambda_n(Y_*)}{\lambda_1(R) \lambda_1^2(\Sigma)} \|\nabla f_{S}(K)\|_F^2.
	\]
\end{proof}
\begin{lemma}\label{lemma:upper_bound_Y}
For \(K \in \mathcal{S}\) and the solution to the Lyapunov matrix equation 
\[
	A_{K} Y + Y A_{K}^{\top} + \Sigma = 0
\]
it holds that
\begin{equation}\label{eq:upper_bound_Y}
    \lambda_n(Y) \leq \frac{f(K)}{\lambda_1(Q + C^{\top} K^{\top} R K C)}.
\end{equation}
\end{lemma}
\begin{proof}
    \[
    \lambda_1(Q + C^{\top} K^{\top} R K C)\operatorname{Tr}\left(Y\right) \leq \operatorname{Tr}\left(Y\left( Q + C^{\top} K^{\top} R K C  \right)\right) = \operatorname{Tr}\left( X \Sigma \right) = f(K).
    \]
\end{proof}

\begin{lemma}\label{lemma:K_norm_bound}
For \(K\in \mathcal{S}\) the norm \(\|K\|_F\) is bounded for \(f(K)\) bounded and
\begin{equation}\label{eq:K_norm_bound}
\|K\|_F  \leq \frac{2\|B\| f(K)}{\lambda_1(\boldsymbol{\Sigma}) \lambda_{1}(R)} + \frac{\|A\|}{\|B\|}.
\end{equation}
\end{lemma}
\begin{proof}
Indeed, consider \cref{lowerbound2} as a quadratic equation with respect to \(\|K\|_F\). Bounding its largest root we obtain an explicit expression
\[
\|K\|_F \leq \frac{2\|B\| f(K) + 2\|B\|f(K)\sqrt{1+ \frac{2\|A\| \lambda_1(\boldsymbol{\Sigma}) \lambda_{1}(R)}{ \|B\|^2f(K)}}}{2\lambda_{1}(\boldsymbol{\Sigma}) \lambda_{1}(R)} \leq \frac{2\|B\| f(K)}{\lambda_1(\boldsymbol{\Sigma}) \lambda_{1}(R)} + \frac{\|A\|}{\|B\|}.
\]
\end{proof}
\subsection{Proof of \cref{thm:LPL}}
\begin{proof}
In view of \cref{lemma:almost_LPL} it sufficies to plug \cref{eq:K_norm_bound,eq:upper_bound_Y} into \cref{eq:almost_LPL}.
\end{proof}

\section{Analysis of the Methods}\label{appendix:methods}

\subsection{Proof of \cref{thm:cont}}
\begin{proof}
The proof is the direct replica of Theorems 8, 9 in \cite{LPL_Polyak}. The only difference is that in \cite{LPL_Polyak} the objective function was defined on the entire space while here it is defined on \(\mathcal{S}\) and is $L$-smooth on $\mathcal{S}_0\in \mathcal{S}$. But
differentiating $f(K_t)$ as a function of $t$ we get $\frac{d}{dt}f(K_t)=-||\nabla f(K_t)||^2$, thus $f(K_t)$ is monotone and $K(t)$  remains in $\mathcal{S}_0$ for all $t\ge 0$. The estimate  (\ref{cont1}) follows from
\[f(K_0)\ge f(K_0) -f(K_{T}) = \int_{0}^{T}||\nabla f(K_t)||^2 dt\ge T\min_{0\le t\le T} ||\nabla f(K_t)||^2.\]
\end{proof}

\subsection{Proof of \cref{thm:grad}}
\begin{proof}
Denote $K=K_0$ and introduce $\varphi(t) = f(K_t), K_t=K-t \nabla f(K)$ as in (\ref{varphi}); we assume $\nabla f(K)\neq 0$. Then scalar function $\varphi(t)$ is differentiable for small $t$ (because $K$ is the interior  point of $ \mathcal{S}$ and $f(K)$ is differentiable on $ \mathcal{S}$) and $\varphi '(0)=-\|\nabla f(K)\|^2<0.$ Thus  $\varphi(t)<\varphi(0)=f(K)$ and $K_t\in  \mathcal{S}_0$ for small $t>0$. The set  $\mathcal{S}_0$ is bounded, denote $T=\max \{t: K_{\tau}\in  \mathcal{S}_0, 0\le\tau\le t\}$ (i.e. the moment of the first intersection of the ray $K_t, t>0$ with the boundary of $ \mathcal{S}_0$). For $0\le t\le T$ we can exploit \cref{thm:Lipschitz_const}  to guarantee $L$-smoothness of $f(K_t)$. This implies  $L'$-smoothness of $\varphi(t), 0\le t \le T$ with $L'=L\|\nabla f(K)\|^2$. Hence $|\varphi(T)-\varphi(0)-\varphi '(0)T|\le \frac{L' T^2}{2}$. But $\varphi(T)=\varphi(0)=f(K)$ and we conclude that $T\ge \frac{2}{L}$. This means that $K_t=K-t \nabla f(K)\in  \mathcal{S}_0$ for all $0\le t \leq \frac{2}{L}$. Thus the segment $ [K_0, K_1]\in  \mathcal{S}_0$. The same follows for all $K_j, j>1$. The end of the proof is the same as for Theorems 3, 4 in \cite{LPL_Polyak}, because we are able to use $L$-smoothness along the entire descent trajectory. 

\end{proof}

\subsection{Proof of \cref{thm:alg}}
The proof can be easily reconstructed via the comments which led to the formulation of the algorithm.

\subsection{Proof of \cref{thm:general_local}}
\begin{proof}
Lipschitz continuity of \(\nabla^2 f\) implies
\[
\left|f(x+y) - f(x) - \langle \nabla f(x), y \rangle - \frac{1}{2} \langle \nabla^2 f(x) y, y \rangle \right| \leq \frac{M}{6}\|y\|^3, \quad \forall x, y \in \mathbb{R}^n.
\]
Taking \(x = x_j, y = - \gamma_j \nabla f(x_j)\) yields
\[
\left|f(x_{j+1}) - f(x_j) + \frac{1}{2} \gamma_j \| \nabla f(x_j)\|^2 \right| \leq \frac{M \gamma_j^3}{6} \|\nabla f(x_j)\|^3.
\]
Denoting \(\varphi_{j} = f(x_j)\) we obtain
\[
\varphi_{j+1} \leq \varphi_{j} - \frac{1}{2}\gamma_j \|\nabla f(x_j)\|^2 \left(1 - \frac{M \gamma_j^2}{3} \|\nabla f(x_j)\| \right).
\]
Now it is clear that if $M \|\nabla f(x_j)\| < 3\mu^2$, then $\varphi_{j+1} < \varphi_{j}$. It follows from $\frac{1}{2}\|\nabla f(x_j)\|^2 \leq L \varphi_{j}$ that this condition is satisfied at each step if (\ref{local_condition}) holds. Then the sequence $\{ \varphi_j\}_{j=0}^{\infty}$ is strictly monotonic decreasing and in view of $\mu$-strong convexity we obtain
\[
\varphi_{j+1} \leq \varphi_{j} \left( 1 - \frac{\mu \delta}{L}\right).
\]

We proceed to proving global convergence of the damped method.
For $\mu$-strongly convex and $L$-smooth function it holds for any $x, y \in \mathbb{R}^n$
\begin{equation}\label{UP}
    f(y) \leq f(x)+\langle\nabla f(x), y-x\rangle+\frac{L}{2\mu}\|y-x\|_{\nabla^{2} f(x)}^{2}.
\end{equation}
Taking into account that $\sigma \leq \frac{\mu}{L}$ and applying (\ref{UP}) for $x = x_j$ and $y = x_j - \sigma \gamma_j \nabla f(x_j)$ we obtain
\begin{equation}
\begin{split}
    f\left(x_{j+1}\right) &\leq f\left(x_{j}\right)+\left\langle\nabla f\left(x_{j}\right), x_{j+1}- x_j\right\rangle+\frac{1}{2\sigma}\left\| x_{j+1} - x_{j}\right\|_{\nabla^{2} f\left(x_{j}\right)}^{2} \\&
    \leq f\left(x_{j}\right) - \sigma \gamma_j \|\nabla f(x_j)\|^2 + \frac{\sigma \gamma_j^2}{2}\|\nabla f(x_j)\|_{\nabla^{2} f\left(x_{j}\right)}^2 = f\left(x_{j}\right) - \frac{\sigma\gamma_j}{2}\|\nabla f(x_j)\|^2.
\end{split}
\end{equation}
Further $L$-smoothness and $\mu$-strong convexity ensure that
\[
f\left(x_{j+1}\right)-f\left(x_{*}\right) \leq\left(f\left(x_{j}\right)-f\left(x_{*}\right)\right)\left(1-\frac{\mu\sigma}{L}\right).
\]
\end{proof}
\section*{Acknowledgement} The authors are grateful to two anonymous reviewers for their helpful comments and to Bin Hu for detecting the gap in the proof of Theorem 4.2.
\bibliographystyle{siamplain}
\bibliography{revised}

\end{document}